\documentclass{amsart}

\usepackage{amsmath,amssymb,amsthm,mathtools}
\usepackage{enumitem}
\usepackage{iftex}
\ifptex
  \usepackage[dvipdfmx,hidelinks]{hyperref}
\else
  \usepackage[hidelinks]{hyperref}
\fi
\usepackage{cleveref}

\theoremstyle{plain}
\newtheorem{theorem}{Theorem}[section]
\newtheorem{proposition}[theorem]{Proposition}
\newtheorem{lemma}[theorem]{Lemma}
\newtheorem{corollary}[theorem]{Corollary}
\crefname{theorem}{Theorem}{Theorems}
\crefname{proposition}{Proposition}{Propositions}
\crefname{lemma}{Lemma}{Lemmas}
\crefname{corollary}{Corollary}{Corollaries}

\theoremstyle{definition}
\newtheorem{definition}[theorem]{Definition}
\crefname{definition}{Definition}{Definitions}

\theoremstyle{remark}

\crefname{remark}{Remark}{Remarks}
\crefname{section}{Section}{Sections}
\crefname{equation}{Equation}{Equations}

\newcommand{\R}{\mathbb{R}}
\newcommand{\D}{\mathcal{D}}
\newcommand{\X}{\mathcal{X}}
\newcommand{\Xp}{\mathcal{X}^{+}}
\newcommand{\M}{\mathcal{M}}
\newcommand{\DM}{\mathcal{DM}}
\renewcommand{\L}{\mathcal{L}}
\newcommand{\K}{\mathcal{K}}
\newcommand{\TB}{\mathcal{TB}}
\newcommand{\Lup}{\operatorname{Lip}^{+}_{1}(\R)}
\newcommand{\lipplus}{\operatorname{Lip}^{+}_{1}}
\DeclareMathOperator{\lipone}{Lip_1}
\DeclareMathOperator{\dis}{dis}
\DeclareMathOperator{\supp}{supp}
\DeclareMathOperator{\pr}{pr}
\DeclareMathOperator{\pd}{PartDiam}
\DeclareMathOperator{\od}{ObsDiam}
\newcommand{\dpr}{d_{\operatorname{P}}}
\newcommand{\dinf}[1]{d^{#1}_{\infty}}
\newcommand{\haus}[1]{\left(#1\right)_H}
\newcommand{\Tplan}{\mathcal{T}}

\title[One-Sided Box Geometry and Pyramid Invariants]{Extensions of One-Sided Box Geometry and Pyramid Invariants to gd-Sets and qm-Spaces}
\author{Shigeaki Yokota}
\date{}

\begin{document}

\begin{abstract}
Domination of gd-sets, the relation recording which function family
approximates which, is closed under box convergence.  More generally,
approximate $1$-Lipschitz maps on asymptotically full-measure subsets, whose
restricted measures converge to the target measure after pushforward, place
the target representation in every subsequential weak limit of the source
pyramids.  For weakly convergent pyramids, bounded joint distributions of
finite ordered tuples of observables recover both observable diameter, after
rightward perturbation of its mass parameter, and the largest common forward
gap among several positive-mass sets, after common leftward perturbation of
their mass parameters.  The lower- and upper-limit formulas agree as the
perturbations vanish.  Without closure assumptions on a gd-set's function
family, we compare observable diameter, the nonnegative part of forward
separation, and upper and lower median-tail masses.  All observables become
uniformly close in measure to suitable constants exactly when both median-tail
masses vanish at every positive radius.
\end{abstract}

\maketitle

\section{Introduction}
\label{sec:introduction}

Concentration on high-dimensional unit spheres is the classical starting point of metric measure geometry: every $1$-Lipschitz observable is nearly constant on almost all of the sphere.  Gromov developed the geometry of mm-spaces to treat this phenomenon while retaining both the distance and the measure.  The distance determines which $1$-Lipschitz observables exist, and the measure determines whether they are nearly constant \cite[Introduction]{shioya2016mmg}.  Gromov introduced the observable distance and the stronger box distance to express concentration as convergence, and proposed a natural compactification \cite[Introduction and Section~2.2]{esaki-kazukawa-mitsuishi2024cones}.  In its later pyramid formulation, an mm-space is sent to the box-closed family of everything it dominates \cite[Definition~2.24]{esaki-kazukawa-mitsuishi2024cones}.  Shioya metrized this topology and made the space of pyramids compact \cite[Definition~4.5 and Theorem~4.6]{shioya2022sugaku}.  The objects here are geometric data sets, whose chosen function families replace the full class of Lipschitz observables, and directed pyramids, whose domination order remembers which family approximates which.

Vanishing ordinary box distance after symmetrization need not make directed
measured structures identical or even comparable.  Let
$Q=\{a,b\}$ carry masses $\mu(a)=1/3$ and $\mu(b)=2/3$, with
$d(a,b)=0$ and $d(b,a)=1$, and let $Q^{\mathrm{op}}$ have the reversed
directed distances.  Their symmetrizations coincide, so their ordinary box
distance is zero.  The unequal atom masses leave the identity as the only
measure-preserving map, but it violates nonexpansion on $(a,b)$ from
$Q$ to $Q^{\mathrm{op}}$ and on $(b,a)$ in the reverse direction.  Thus
neither space dominates the other.  A one-point space carries no asymmetry,
so this witness is minimal in cardinality.

A triple
$(X,F_X,\mu_X)$ is a geometric data set, or gd-set, if $F_X$ is a nonempty
family of real-valued functions on $X$,
$d_{F_X}(x,x')\coloneqq\sup_{f\in F_X}|f(x)-f(x')|$ defines a complete
separable metric on $X$, and $\mu_X$ is
a full-support Borel probability measure \cite[Definition~3.1]{gds1}.
Write $\overline{F_X}$ for the pointwise closure of $F_X$.  A Borel map
$u\colon X\to Y$ is a domination if
$u_*\mu_X=\mu_Y$ and
$F_Y\circ u\subset\overline{F_X}$, and in this case write $Y\preceq X$.
This order records which function family approximates which.

Approximate domination requires the same direction on a high-measure closed
relation.  Write $\Box(X,Y)$ for the ordinary box distance.  For
$\varepsilon\geq0$, write $Y\preceq_\varepsilon X$ if a coupling
$\pi\in\mathcal T(\mu_X,\mu_Y)$ and a closed set $S\subset X\times Y$
satisfy $\pi(S)\geq1-\varepsilon$ and
\[ 2\sup_{g\in\overline{F_Y}}\inf_{f\in\overline{F_X}}\sup_{(x,y)\in S}|g(y)-f(x)|\leq\varepsilon, \]
where the innermost supremum is $0$ for $S=\varnothing$.  This is domination
with additive error $\varepsilon$, and the one-sided box distance from $X$
to $Y$ is
\begin{equation}\label{eq:unilateral-box-distance}
\Box_{\preceq}(Y,X)\coloneqq\inf\{\varepsilon\geq0\mid Y\preceq_\varepsilon X\}.
\end{equation}

The same repair applies to directed distances.  A qm-space
$(X,d_X,\mu_X)$ has a directed distance whose symmetrization is a complete
separable metric and a full-support Borel probability measure.  Its
one-sided $1$-Lipschitz functions are
\[ \lipplus(X)\coloneqq\{f\colon X\to\R\mid f(y)-f(x)\leq d_X(x,y)\text{ for all }x,y\in X\}, \]
This semi-Lipschitz class was introduced for quasi-metric spaces in
\cite{romaguera2000semi}.  The gd-set
$\operatorname{Rep}^{+}(X)\coloneqq(X,\lipplus(X),\mu_X)$ recovers the
distance from their ordered increments.  A pyramid is a nonempty,
box-closed family that is downward closed under domination and in which
every two objects are dominated by a third \cite{gds2}.  Weak
convergence of pyramids means sequential Painlev\'e--Kuratowski convergence,
also called weak Hausdorff convergence in mm-space theory.  In the
applications the directed maps are available only on a subset whose measure
tends to one, so part \textup{(ii)} below is stated in that form.

\begin{theorem}[Directed stability and mass-perturbed limits]
\label{thm:intro-directed-stability}
\begin{enumerate}[label=\textup{(\roman*)}]
\item For gd-sets $X,Y,Z$,
\[ \Box_{\preceq}(Y,X)\leq\Box(X,Y),\qquad \Box_{\preceq}(Z,X)\leq\Box_{\preceq}(Z,Y)+\Box_{\preceq}(Y,X), \]
and $\Box_{\preceq}(Y,X)=0$ if and only if $Y\preceq X$.  If gd-sets
$X_n,Y_n,X,Y$ and numbers $\delta_n\geq0$ satisfy
\begin{align*}
\delta_n&\longrightarrow0,&
Y_n&\preceq_{\delta_n}X_n,\\
\Box(X_n,X)&\longrightarrow0,&
\Box(Y_n,Y)&\longrightarrow0,
\end{align*}
then $Y\preceq X$.
\item Suppose that Borel maps from a sequence of qm-spaces to a fixed
qm-space are defined only on Borel subsets of the individual source spaces.
If these subsets have measures tending to $1$, the pushforwards of the
restricted measures converge weakly to the target measure as finite Borel
measures, and the maps do not increase the directed distance between any
ordered pair except by an additive error tending to zero, then the target,
represented as a gd-set, belongs to every subsequential weak limit of the
pyramids generated by the represented source spaces.  The precise statement
is given in \Cref{prop:high-mass-approximate-domination}.
\item When a sequence of pyramids converges weakly, the observable diameter
of the limit pyramid is recovered by moving its single mass parameter to the
right by a positive error, taking either the lower or the upper limit along
the sequence, and then letting the error decrease to zero.  The lower-limit
and upper-limit versions give the same value.  For multi-directed separation,
all mass parameters are instead moved simultaneously to the left by the same
positive error before the same operations are performed.  In this case as
well, both versions agree with the value on the limit pyramid.  The precise
formulas are given in
\Cref{prop:lup-observable-diameter-pyramid-limit,prop:tb-multi-directed-separation-pyramid-limit}.
\end{enumerate}
\end{theorem}

Assigning one coordinate to each set pair $i<j$ holds all ordered gaps in a
single finite measurement.  One Prokhorov error reduces the guaranteed mass
of every witness set at once, which is why a common leftward perturbation
appears rather than independent ones.

In the classical concentration theory of mm-spaces,
$\operatorname{Lip}_1$ is closed under negation, so the upper and lower
tails can be combined into one concentration function, while distance
functions from sets yield observables for comparing separation with
observable diameter.  Neither property is available for a general gd-set:
$\overline{F_X}$ need not be closed under negation, and distance functions
from sets need not belong to $\overline{F_X}$.  We therefore let
$\alpha_X^+(r)$ and $\alpha_X^-(r)$ record the supremal masses of the upper
and lower tails, respectively, at distance at least $r$ from a median, as the
function in $\overline{F_X}$ and its median vary.  Without any additional
closure assumption on the function family,
\Cref{sec:one-sided-concentration} compares these two one-sided concentration
functions with observable diameter and nonnegative directed separation.  We
call a sequence of gd-sets a function-family L\'evy family when every
function becomes close in measure to a suitable constant uniformly over the
function family; this is equivalent to
$\alpha_{X_n}^+(r)\to0$ and $\alpha_{X_n}^-(r)\to0$ for every $r>0$.

The dependence on the companion paper \cite{gds3-ja} is one-way: its abstract
pyramid transport results are used here, while they do not use the
quantitative estimates proved in this paper.

The following list records every fixed statement from the upstream papers
\cite{gds1,gds2,gds3-ja} used in this manuscript.

\begingroup\raggedright
\begin{description}
\item[\cite{gds1}, Definition~3.1]
\textup{Use and first use.} Supplies the gd-set triple, its induced complete
separable metric, and its full-support probability measure in the opening
gd-set paragraph of this introduction.

\item[\cite{gds1}, Definition~3.8]
\textup{Use and first use.} Supplies domination and the feature order in the
opening gd-set paragraph of this introduction.

\item[\cite{gds1}, Lemma~4.11]
\textup{Use and first use.} Gives the Prokhorov perturbation estimate for
partial diameter in the proof of
\Cref{prop:lup-observable-diameter-pyramid-limit}.

\item[\cite{gds1}, Lemma~4.13]
\textup{Use and first use.} Gives right-continuity of partial diameter, used
to obtain right-continuity of pyramid observable diameter in the proof of
\Cref{prop:lup-observable-diameter-pyramid-limit}.

\item[\cite{gds1}, Definition~5.7]
\textup{Use and first use.} Supplies the gd-set box distance adopted in
\Cref{def:gds-distances} and used as the ambient distance thereafter.

\item[\cite{gds1}, Proposition~5.13]
\textup{Use and first use.} Makes the box distance a metric on gd-set
isomorphism classes in the paragraph following \Cref{def:gds-distances}.

\item[\cite{gds2}, Proposition~3.14]
\textup{Use and first use.} Makes the domination-refinement lemma applicable
to the monoidal families used in the coarsening argument after
\Cref{prop:coarsening-box}.

\item[\cite{gds2}, Theorem~3.19]
\textup{Use and first use.} Identifies the closed gd-sets used here with the
upstream compact class to which the measurement results apply before
\Cref{lem:pyramid-measurement-compactness}.

\item[\cite{gds2}, Definition~4.3]
\textup{Use and first use.} Supplies the unordered feature-measurement sets
used in
\Cref{lem:pyramid-measurement-compactness,lem:pyramid-measurement-convergence}.

\item[\cite{gds2}, Lemma~4.6]
\textup{Use and first use.} Gives box-nonexpansiveness of monoidal saturation
in the proof of \Cref{prop:coarsening-box}.

\item[\cite{gds2}, Lemmas~4.8 and~4.9]
\textup{Use and first use.} Approximate an object in the upstream compact
class in box distance by monoidal saturations of bounded finite-feature
models in the reconstruction part of
\Cref{lem:pyramid-measurement-convergence}.

\item[\cite{gds2}, Definition~5.1]
\textup{Use and first use.} Supplies the downward-closed, directed,
nonempty, and box-closed pyramid axioms first stated in this introduction.

\item[\cite{gds2}, Lemma~5.4]
\textup{Use and first use.} Supplies domination refinement for the working
gd-set categories, used after \Cref{prop:coarsening-box} to verify that the
inclusion--saturation adjunction is pyramidal.

\item[\cite{gds2}, Lemma~7.1]
\textup{Use and first use.} Gives compactness, and hence closedness, of the
unordered measurement set in the proof of
\Cref{lem:pyramid-measurement-compactness}.

\item[\cite{gds2}, Proposition~7.2]
\textup{Use and first use.} Detects weak convergence of pyramids through
unordered finite measurements in the proof of
\Cref{lem:pyramid-measurement-convergence}.

\item[\cite{gds3-ja}, Definition~A.5]
\textup{Use and first use.} Supplies the criteria used after
\Cref{prop:coarsening-box} to recognize the inclusion--saturation adjunction
as pyramidal.

\item[\cite{gds3-ja}, Theorem~A.7]
\textup{Use and first use.} Gives pyramid transport and the equivalence of
weak convergence, first used after \Cref{prop:coarsening-box} and later used
to pull the directional-invariant limit formulas back along pyramidal
adjunctions.

\item[\cite{gds3-ja}, Proposition~A.8]
\textup{Use and first use.} Gives the pyramidal property and transport
compatibility for finite composites, first used after
\Cref{prop:coarsening-box} and later used for the composites with the
representation adjunctions.

\item[\cite{gds3-ja}, Corollary~A.12]
\textup{Use and first use.} Makes the representation--reconstruction
adjunctions pyramidal in the specialization of
\Cref{thm:observable-diameter-pyramid-limit} to mm-spaces and qm-spaces.
\end{description}
\endgroup

\section{Measured spaces and function families}
\label{sec:background}

The conventions in this section supply the couplings and oscillation
pseudometrics used to compare exact and approximate domination in
\Cref{sec:one-sided-box}.

We denote the category of gd-set isomorphism classes and dominations by
$\D$.

Write $\Tplan(\mu,\nu)$ for the set of couplings of two Borel probability
measures.  For a closed set $S\subset X\times Y$ and functions on
$X\times Y$, put
\[
\dinf{S}(h,k)\coloneqq\sup_{(x,y)\in S}|h(x,y)-k(x,y)|
\]
and let $\haus{\dinf{S}}$ be the induced Hausdorff pseudometric on function
families.  Both are set equal to $0$ when $S$ is empty.

\begin{definition}[Box distance and one-sided distortion]
\label{def:gds-distances}
For gd-sets $X,Y$ and a closed set $S\subset X\times Y$, define
\begin{equation}
\label{eq:gds-directed-distortion}
\dis_{\succ}S\coloneqq
2\sup_{g\in\overline{F_Y}}\inf_{f\in\overline{F_X}}
\dinf{S}(g\circ\pr_2,f\circ\pr_1).
\end{equation}
The box distance is
\begin{equation}
\label{eq:gds-box}
\Box(X,Y)\coloneqq
\inf\left\{\eta\geq0\ \middle|\
\begin{array}{l}
\pi\in\Tplan(\mu_X,\mu_Y),\\
S\subset X\times Y\text{ is closed},\\
1-\pi(S)\leq\eta,\\
2\haus{\dinf{S}}(\overline{F_X}\circ\pr_1, \overline{F_Y}\circ\pr_2)\leq\eta
\end{array}\right\}.
\end{equation}
\end{definition}

The box distance in \Cref{eq:gds-box} is a distance on gd-set isomorphism
classes \cite[Definition~5.7 and Proposition~5.13]{gds1}.

\section{One-sided box geometry}
\label{sec:one-sided-box}

High-measure closed relations compose after their couplings are glued, and
their missing masses and one-sided oscillation errors add.  At zero error,
the support of a limiting coupling is the graph of an exact domination.
These two facts give the triangle inequality and the zero criterion below.

\begin{definition}[One-sided box notation]
\label{def:additive-error-domination}
For gd-sets $X,Y$, the additive-error domination relation
$Y\preceq_\varepsilon X$ and the one-sided box distance
$\Box_{\preceq}(Y,X)$ are those introduced in
\Cref{sec:introduction,eq:unilateral-box-distance}.  The function-family
bound in that definition is $\dis_{\succ}S\leq\varepsilon$ in the notation of
\Cref{eq:gds-directed-distortion}.
\end{definition}

\begin{proposition}[Comparison with the box distance]
\label{prop:unilateral-box-comparison}
For any gd-sets $X,Y$,
\[ \Box_{\preceq}(Y,X)\leq\Box(X,Y),\qquad \Box_{\preceq}(X,Y)\leq\Box(X,Y). \]
\end{proposition}

\begin{proof}
The directed Hausdorff error in \Cref{eq:gds-directed-distortion} is at
most the Hausdorff error in \Cref{eq:gds-box}.  Taking the infimum over the
same couplings and closed sets gives the first inequality.  Interchanging
$X$ and $Y$ gives the second.
	This completes the proof.
\end{proof}

\begin{proposition}[Triangle inequality]
\label{prop:unilateral-box-triangle}
For any gd-sets $X,Y,Z$,
\[ \Box_{\preceq}(Z,X)\leq \Box_{\preceq}(Z,Y)+\Box_{\preceq}(Y,X). \]
\end{proposition}

\begin{proof}
If the right-hand side is at least $1$, the assertion follows from
$\Box_{\preceq}(Z,X)\leq1$.  Suppose that $Y\preceq_sX$ and
$Z\preceq_tY$ with $s+t<1$, and choose the corresponding couplings and
closed sets $\sigma,S$ and $\tau,T$.  By the gluing lemma for couplings
\cite{villani2009optimal}, take a probability measure $\eta$ on
$X\times Y\times Z$ whose $(X,Y)$-marginal is $\sigma$ and whose
$(Y,Z)$-marginal is $\tau$.  Define
\[
\theta\coloneqq(\pr_{13})_*\eta,\qquad
U\coloneqq\overline{\pr_{13}
(\{(x,y,z)\mid(x,y)\in S,\ (y,z)\in T\})}.
\]
Then
\[
\theta(U)\geq\eta((S\times Z)\cap(X\times T))\geq1-s-t.
\]

Take $h\in\overline{F_Z}$ and $\delta>0$.  Since
$\dis_{\succ}T\leq t$ and $\dis_{\succ}S\leq s$, choose
$g\in\overline{F_Y}$ and $f\in\overline{F_X}$ such that
\begin{align*}
\dinf{T}(h\circ\pr_2,g\circ\pr_1)&<t/2+\delta,\\
\dinf{S}(g\circ\pr_2,f\circ\pr_1)&<s/2+\delta.
\end{align*}
On the composed relation, the difference between $h$ and $f$ is less than
$(s+t)/2+2\delta$.  Continuity gives the corresponding non-strict estimate
on $U$.  Taking the supremum over $h$ and letting $\delta\downarrow0$ gives
$\dis_{\succ}U\leq s+t$.  Hence $Z\preceq_{s+t}X$.  Taking infima over
$s$ and $t$ proves the assertion.  This completes the proof.
\end{proof}

\begin{proposition}[The zero set]
\label{prop:unilateral-box-zero}
For any gd-sets $X,Y$,
\[
Y\preceq X\quad\Longleftrightarrow\quad
Y\preceq_0X\quad\Longleftrightarrow\quad
\Box_{\preceq}(Y,X)=0.
\]
\end{proposition}

\begin{proof}
Suppose first that $Y\preceq X$, and let $u\colon X\to Y$ be a
domination.  Its graph $S$ is closed because $u$ is $1$-Lipschitz for the
induced metrics.  Moreover,
$\overline{F_Y}\circ u\subset\overline{F_X}$.  Thus
$(\operatorname{id}_X,u)_*\mu_X$ and $S$ show that $Y\preceq_0X$.

Suppose that $Y\preceq_0X$, and choose a coupling $\pi$ and a closed set
$S$ realizing this relation.  Set $R\coloneqq\supp\pi$.  Then $R\subset S$
and $\dis_{\succ}R=0$.  For
$(x,y),(x',y')\in R$, $g\in\overline{F_Y}$, and $\delta>0$, choose
$f\in\overline{F_X}$ with
\[
\dinf{R}(g\circ\pr_2,f\circ\pr_1)<\delta.
\]
It follows that
\[
|g(y)-g(y')|\leq d_{F_X}(x,x')+2\delta.
\]
Taking the supremum over $g$ and letting $\delta\downarrow0$ gives
$d_{F_Y}(y,y')\leq d_{F_X}(x,x')$.  The first projection of $R$ is dense
in $X$.  If $x_m\to x$ and $(x_m,y_m)\in R$, this estimate makes
$(y_m)$ Cauchy.  Completeness of $Y$ and closedness of $R$ give a unique
$y$ with $(x,y)\in R$.  Thus $R$ is the graph of a $1$-Lipschitz map
$u\colon X\to Y$.  For $g\in F_Y$, the functions chosen above along
$\delta\downarrow0$ converge uniformly to $g\circ u$, so
$F_Y\circ u\subset\overline{F_X}$.  The second marginal of $\pi$ gives
$u_*\mu_X=\mu_Y$.  Hence $Y\preceq X$.

Finally suppose that $\Box_{\preceq}(Y,X)=0$.  Take
$\varepsilon_n\downarrow0$, couplings $\pi_n$, and closed sets $S_n$ such that
$Y\preceq_{\varepsilon_n}X$.  Tightness gives a subsequence converging weakly
to some $\pi\in\Tplan(\mu_X,\mu_Y)$.  Set $R\coloneqq\supp\pi$.  The
Portmanteau theorem and $1-\pi_n(S_n)\leq\varepsilon_n$ imply that for every
$(x,y)\in R$ there are points
\begin{equation}
\label{eq:unilateral-box-zero-support-approximation}
(x_n,y_n)\in S_n,\qquad (x_n,y_n)\longrightarrow(x,y).
\end{equation}
Fix $g\in\overline{F_Y}$ and choose $f_n\in\overline{F_X}$ with
\[ \dinf{S_n}(g\circ\pr_2,f_n\circ\pr_1) <\frac{\varepsilon_n}{2}+\frac1n. \]
Applying \Cref{eq:unilateral-box-zero-support-approximation} at one point
of $R$ bounds the values of $f_n$ there.  Separability and a diagonal
subsequence argument give pointwise convergence to some
$f\in\overline{F_X}$.  Applying the same approximation at any
$(x,y)\in R$ gives $f(x)=g(y)$.  Hence $\dis_{\succ}R=0$ and
$Y\preceq_0X$.  The equivalence already proved yields $Y\preceq X$.
This completes the proof.
\end{proof}

\begin{proposition}[Additive-error domination and box limits]
\label{prop:additive-domination-box-closedness}
Suppose that gd-sets $X_n,Y_n,X,Y$ and numbers $\delta_n\geq0$ satisfy
\[
\delta_n\longrightarrow0,\qquad Y_n\preceq_{\delta_n}X_n,\qquad
\Box(X_n,X)\longrightarrow0,\qquad\Box(Y_n,Y)\longrightarrow0.
\]
Then $Y\preceq X$.  In particular, domination of gd-sets is box-closed.
\end{proposition}

\begin{proof}
By \Cref{prop:unilateral-box-comparison,prop:unilateral-box-triangle},
\[ 0\leq\Box_{\preceq}(Y,X) \leq\Box(Y,Y_n)+\delta_n+\Box(X_n,X)\longrightarrow0. \]
The conclusion follows from \Cref{prop:unilateral-box-zero}.
	This completes the proof.
\end{proof}

For the qm-space specialization, the directed distance is part of the
measured structure and is recovered from its one-sided Lipschitz functions.

\begin{definition}[Asymmetric metric spaces and qm-spaces]
\label{def:asymmetric-mm-space}
The pair $(X,d_X)$ is called an \emph{asymmetric metric space} if
$d_X\colon X\times X\to[0,+\infty)$ satisfies $d_X(x,x)=0$ and the triangle
inequality, and the symmetrization
\[
d_X^{\mathrm s}(x,y)\coloneqq\max\{d_X(x,y),d_X(y,x)\}
\]
is a metric on $X$.  A triple $(X,d_X,\mu_X)$ is called a \emph{qm-space}
if $(X,d_X)$ is an asymmetric metric space, $d_X^{\mathrm s}$ is complete
and separable, and $\mu_X$ is a Borel probability measure on
$(X,d_X^{\mathrm s})$ with $\supp\mu_X=X$.
\end{definition}

A measure-preserving map $u\colon X\to Y$ between qm-spaces is
\emph{$1$-Lipschitz} if
\[
d_Y(u(x),u(x'))\leq d_X(x,x')\qquad(x,x'\in X).
\]
We write $Y\preceq X$ when such a map exists.  An mm-space is a qm-space
whose distance is symmetric.  Its domination order is the usual
measure-preserving $1$-Lipschitz order.  We write $\Xp$ and $\X$ for the
respective categories of isomorphism classes.

\begin{definition}[One-sided $1$-Lipschitz functions]
\label{def:lipplus}
For a qm-space $X$, the one-sided $1$-Lipschitz class $\lipplus(X)$ and its
gd-set representation $\operatorname{Rep}^{+}(X)$ are those introduced in
\Cref{sec:introduction}.  For an mm-space $X$, set
$\operatorname{Rep}(X)\coloneqq(X,\lipone(X),\mu_X)$.
\end{definition}

The ordered increment recovers the directed distance:
\begin{equation}
\label{eq:directed-recovery}
d_X(x,y)=\sup_{f\in\lipplus(X)}\{f(y)-f(x)\}.
\end{equation}
Indeed, one inequality follows from the definition and the other from
$z\mapsto d_X(x,z)$.  When $X$ is an mm-space,
\begin{equation} \label{eq:symmetric-directed-lipschitz} \lipplus(X)=\lipone(X). \end{equation}

For qm-spaces $X,Y$ and a closed set $S\subset X\times Y$, define
\[ \dis_{\succ}S\coloneqq \sup\{d_Y(y,y')-d_X(x,x')\mid(x,y),(x',y')\in S\}, \]
with value $0$ for the empty set.  The definitions of
$Y\preceq_\varepsilon X$ and $\Box_{\preceq}(Y,X)$ are those in
\Cref{def:additive-error-domination}, with the displayed directed-distance
distortion in place of the function-family oscillation.

\begin{proposition}[Coincidence of the two one-sided distortions]
\label{prop:directed-distortion-coincidence}
Let $X_0,Y_0$ be qm-spaces and put
$X\coloneqq\operatorname{Rep}^{+}(X_0)$ and $Y\coloneqq\operatorname{Rep}^{+}(Y_0)$.  For
every nonempty closed set $S\subset X_0\times Y_0$,
\begin{equation}
\label{eq:directed-distortion-coincidence}
\dis_{\succ}S=
\sup\{d_{Y_0}(y,y')-d_{X_0}(x,x')\mid
(x,y),(x',y')\in S\},
\end{equation}
where the left-hand side is the gd-set distortion in
\Cref{eq:gds-directed-distortion}.  Consequently, for every
$\varepsilon\geq0$,
\[
\begin{aligned}
Y_0\preceq_\varepsilon X_0
&\quad\Longleftrightarrow\quad
\operatorname{Rep}^{+}(Y_0)\preceq_\varepsilon\operatorname{Rep}^{+}(X_0),\\
\Box_{\preceq}(Y_0,X_0) &=\Box_{\preceq}\bigl(\operatorname{Rep}^{+}(Y_0),\\
&\qquad\operatorname{Rep}^{+}(X_0)\bigr).
\end{aligned}
\]
\end{proposition}

\begin{proof}
The one-sided $1$-Lipschitz condition is preserved under pointwise limits,
so
\[ \overline{F_X}=\lipplus(X_0),\qquad \overline{F_Y}=\lipplus(Y_0). \]
For $g\in\lipplus(Y_0)$, set
\[ r_S(g)\coloneqq \sup_{(x,y),(x',y')\in S} \{g(y')-g(y)-d_{X_0}(x,x')\}. \]
Interchanging the suprema and using \Cref{eq:directed-recovery} gives
\begin{equation}
\label{eq:directed-distortion-first-identity}
\sup_{g\in\lipplus(Y_0)}r_S(g)
=\sup_{(x,y),(x',y')\in S}
\{d_{Y_0}(y,y')-d_{X_0}(x,x')\}.
\end{equation}

For $f\in\lipplus(X_0)$, the one-sided Lipschitz inequality gives
\begin{equation}
\label{eq:directed-distortion-oscillation}
r_S(g)\leq
\sup_{(x,y)\in S}\{g(y)-f(x)\}
-\inf_{(x,y)\in S}\{g(y)-f(x)\}.
\end{equation}
If $r_S(g)<+\infty$, define the inf-convolution
\[
f_g(x)\coloneqq\inf_{(u,v)\in S}\{g(v)+d_{X_0}(u,x)\}.
\]
It is finite-valued and belongs to $\lipplus(X_0)$.  Indeed, a fixed point
of $S$, the definition of $r_S(g)$, and the triangle inequality give a
finite lower bound, while the same triangle inequality gives
$f_g(x')\leq f_g(x)+d_{X_0}(x,x')$.  On $S$ one has
\[
g(y)-r_S(g)\leq f_g(x)\leq g(y).
\]
Together with \Cref{eq:directed-distortion-oscillation}, this shows that
\[ \inf_{f\in\lipplus(X_0)}\dinf{S}(g\circ\pr_2,f\circ\pr_1) =\frac{r_S(g)}2, \]
where the upper bound is attained by $f_g+r_S(g)/2$.  If
$r_S(g)=+\infty$, \Cref{eq:directed-distortion-oscillation} gives the same
identity in the extended real numbers.  Combining this identity with
\Cref{eq:gds-directed-distortion,eq:directed-distortion-first-identity}
proves \Cref{eq:directed-distortion-coincidence}.  Substitution into
\Cref{def:additive-error-domination} proves the final assertions.  This
completes the proof.
\end{proof}

Finite measurements retain bounded coordinate data from the function family.
For a positive integer $N$ and $R>0$, let $\M(N,R)$ be the Borel probability
measures on $[-R,R]^N$, equipped with the Prokhorov distance $\dpr$ induced by
the $\ell^\infty$ metric.  Put
$b_R(t)\coloneqq\max\{-R,\min\{t,R\}\}$ and define
\begin{equation}
\label{eq:ordered-measurement}
\M(X;N,R)\coloneqq
\{(b_R\circ f_1,\ldots,b_R\circ f_N)_*\mu_X\mid
(f_1,\ldots,f_N)\in F_X^N\}.
\end{equation}
For a family $\mathcal E$ of gd-sets, set
\[
\M(\mathcal E;N,R)\coloneqq\bigcup_{Y\in\mathcal E}\M(Y;N,R).
\]
For a subset $A$ of a metric space $(M,d)$, write
\[
\mathrm B(A,r;d)\coloneqq\{x\in M\mid d(x,A)\leq r\}.
\]

\begin{proposition}[Finite-measurement inclusion]
\label{prop:unilateral-box-measurement-inclusion}
For gd-sets $X,Y$, a positive integer $N$, and $R>0$,
\begin{align}
\M(Y;N,R)&\subset
\mathrm B(\M(X;N,R),\Box_{\preceq}(Y,X);\dpr),
\label{eq:unilateral-measurement-inclusion-forward}\\
\M(X;N,R)&\subset
\mathrm B(\M(Y;N,R),\Box_{\preceq}(X,Y);\dpr).
\label{eq:unilateral-measurement-inclusion-backward}
\end{align}
\end{proposition}

\begin{proof}
We prove \Cref{eq:unilateral-measurement-inclusion-forward}.  The assertion
is immediate if $\Box_{\preceq}(Y,X)=1$, so take
$\Box_{\preceq}(Y,X)<\varepsilon<1$.  Choose a coupling $\pi$ and a closed
set $S\subset X\times Y$ such that
\[
1-\pi(S)<\varepsilon,\qquad\dis_{\succ}S<\varepsilon.
\]
Let $g_1,\ldots,g_N\in F_Y$ represent
$\nu\in\M(Y;N,R)$.  For each $j$, choose
$f_j\in\overline{F_X}$ with
\[
\dinf{S}(g_j\circ\pr_2,f_j\circ\pr_1)<\frac\varepsilon2.
\]
Because $b_R$ is $1$-Lipschitz, pushing $\pi$ forward by the two clipped
coordinate maps gives
\[ \dpr\bigl((b_R\circ f_1,\ldots,b_R\circ f_N)_*\mu_X,\nu\bigr) \leq\varepsilon. \]
Approximate each $f_j$ pointwise by elements of $F_X$.  The corresponding
pushforward measures converge weakly, hence in $\dpr$.  Therefore
$\dpr(\nu,\M(X;N,R))\leq\varepsilon$.  Letting
$\varepsilon\downarrow\Box_{\preceq}(Y,X)$ proves the first inclusion.
Interchanging $X$ and $Y$ proves the second.  This completes the proof.
\end{proof}

A pointwise closed subfamily $\L\subset\lipone(\R)$ is \emph{monoidal} if it
contains the identity and is closed under composition.  Set
\[
\TB\coloneqq\{t\mapsto\max\{l,\min\{t+c,u\}\}\mid
c\in\R,\ l\in[-\infty,+\infty),\ u\in(-\infty,+\infty],\ l\leq u\}
\]
and
\[ \Lup\coloneqq\{p\colon\R\to\R\mid p\text{ is nondecreasing and $1$-Lipschitz}\}. \]
Every monoidal family below contains $\TB$.  A gd-set is an
\emph{$\L$-gd-set} if
$\L\circ\overline{F_X}\subset\overline{F_X}$.  The resulting category is
denoted by $\L\circ\D$.

An approximate directed map may be defined only on a high-mass subset, so it
does not directly give a coupling relation as in
\Cref{def:additive-error-domination}.  The inf-convolution used in
\Cref{prop:directed-distortion-coincidence} fills in the missing values, and
finite measurements then transfer the partial map to any pyramid limit.
For $u,v\in\R^N$, put
\[
d_N^+(u,v)\coloneqq\max_{1\leq j\leq N}(v_j-u_j)_+.
\]
The following lemma is a one-sided semi-Lipschitz McShane--Whitney-type
extension in the sense of \cite{romaguera2000semi}.

\begin{lemma}[Prescribed-subset one-sided extension]
\label{lem:prescribed-subset-one-sided-extension}
Let $(X,d_X)$ be an asymmetric metric space, let $A\subset X$ be nonempty,
and let $N$ be a positive integer, $R>0$, and $\varepsilon\geq0$.  Suppose that
a map $f\colon A\to[-R,R]^N$ satisfies
\begin{equation}
\label{eq:prescribed-subset-one-sided-condition}
d_N^+(f(x),f(y))\leq d_X(x,y)+\varepsilon\qquad(x,y\in A).
\end{equation}
Then there is a map $F\colon X\to[-R,R]^N$ such that
\[
d_N^+(F(x),F(y))\leq d_X(x,y)\qquad(x,y\in X)
\]
and
\begin{equation}
\label{eq:prescribed-subset-extension-error}
\sup_{x\in A}\lVert F(x)-f(x)\rVert_\infty\leq\varepsilon.
\end{equation}
Suppose also that $\mu_X$ is a Borel probability measure on
$(X,d_X^{\mathrm s})$, that $A$ is $\mu_X$-measurable, and that $f$ is
measurable on $(A,\mu_X|_A)$.  If $\sigma$ is a Borel probability measure
on $[-R,R]^N$ and
\[
f_*(\mu_X|_A)\leq\sigma,
\]
then $F$ can be chosen to be Borel and to satisfy
\begin{equation}
\label{eq:prescribed-subset-extension-prokhorov}
\dpr(F_*\mu_X,\sigma)\leq
\max\{\varepsilon,1-\mu_X(A)\}.
\end{equation}
In particular, the right-hand side is at most $\varepsilon$ when
$\mu_X(A)\geq1-\varepsilon$.
\end{lemma}

\begin{proof}
Write $f=(f_1,\ldots,f_N)$ and, for $1\leq j\leq N$, define
\[
\widehat f_j(z)\coloneqq\inf_{x\in A}\{f_j(x)+d_X(x,z)\}.
\]
The triangle inequality gives
$\widehat f_j(z')-\widehat f_j(z)\leq d_X(z,z')$.  If $y\in A$, the choice
$x=y$ gives $\widehat f_j(y)\leq f_j(y)$, while
\Cref{eq:prescribed-subset-one-sided-condition} gives
$\widehat f_j(y)\geq f_j(y)-\varepsilon$.  Set
\[
F\coloneqq(b_R\circ\widehat f_1,\ldots,b_R\circ\widehat f_N).
\]
The map $b_R$ is nondecreasing and $1$-Lipschitz.  Therefore $F$ satisfies
the exact one-sided inequality and
\Cref{eq:prescribed-subset-extension-error}.

For the measure estimate, the exact inequality in both orders shows that
$F$ is continuous with respect to $d_X^{\mathrm s}$, and hence Borel.  Set
$\eta\coloneqq\max\{\varepsilon,1-\mu_X(A)\}$.  The measure
$\sigma-f_*(\mu_X|_A)$ is positive and has mass $1-\mu_X(A)$.  Couple it
arbitrarily with $F_*(\mu_X|_{X\setminus A})$, and add this coupling to
$(F,f)_*(\mu_X|_A)$.  The resulting measure couples $F_*\mu_X$ and
$\sigma$, and the paired points are within $\varepsilon$ on mass
$\mu_X(A)\geq1-\eta$.  The coupling characterization of the Prokhorov
distance proves \Cref{eq:prescribed-subset-extension-prokhorov}.  This
completes the proof.
\end{proof}

\begin{proposition}[High-mass approximate domination]
\label{prop:high-mass-approximate-domination}
Let $(X_n,d_n,\mu_n)$ and $(Y,d_Y,\nu)$ be qm-spaces, denoted simply by $X_n$ and $Y$, respectively.  Suppose that
there are Borel sets $A_n\subset X_n$, Borel maps $p_n\colon A_n\to Y$, and
numbers $\varepsilon_n\geq0$ such that
\begin{align}
\varepsilon_n&\longrightarrow0, \label{eq:high-mass-domination-error}\\
\mu_n(A_n)&\longrightarrow1, \label{eq:high-mass-domination-mass}\\
(p_n)_*(\mu_n|_{A_n})&\longrightarrow\nu
\quad\text{weakly as finite Borel measures},
\label{eq:high-mass-domination-measure}\\
d_Y(p_n(x),p_n(y))&\leq d_n(x,y)+\varepsilon_n
\qquad(x,y\in A_n).
\label{eq:high-mass-domination-distance}
\end{align}
Then $\operatorname{Rep}^{+}(Y)$ belongs to every subsequential weak limit of the
pyramids generated by $\operatorname{Rep}^{+}(X_n)$.
\end{proposition}

\begin{proof}
Consider a subsequence along which the generated pyramids converge weakly to
a pyramid $\mathcal P$.  Fix a positive integer $N$ and a number $R>0$.  Take
\[
\lambda\in\M(\operatorname{Rep}^{+}(Y);N,R).
\]
Choose $h_1,\ldots,h_N\in\lipplus(Y)$ and set
$h\coloneqq(b_R\circ h_1,\ldots,b_R\circ h_N)$ so that
$\lambda=h_*\nu$.  The map $h$ takes values in $[-R,R]^N$ and satisfies
\[
d_N^+(h(y),h(y'))\leq d_Y(y,y').
\]

The probability measure
\[ \sigma_n\coloneqq (h\circ p_n)_*(\mu_n|_{A_n})+ (1-\mu_n(A_n))\delta_0 \]
converges weakly to $\lambda$.  Applying
\Cref{lem:prescribed-subset-one-sided-extension} to $h\circ p_n$ gives an
exact one-sided map $F_n\colon X_n\to[-R,R]^N$ such that
\[ \dpr((F_n)_*\mu_n,\sigma_n)\leq \max\{\varepsilon_n,1-\mu_n(A_n)\}\longrightarrow0. \]
Thus $(F_n)_*\mu_n\to\lambda$ in $\dpr$.  Moreover,
\[
(F_n)_*\mu_n\in\M(\operatorname{Rep}^{+}(X_n);N,R).
\]

\Cref{lem:pyramid-measurement-compactness,lem:pyramid-measurement-convergence}
gives $\lambda\in\M(\mathcal P;N,R)$.  This holds for every $N$, $R$, and
$\lambda\in\M(\operatorname{Rep}^{+}(Y);N,R)$.  The reconstruction assertion
in \Cref{lem:pyramid-measurement-convergence} gives
$\operatorname{Rep}^{+}(Y)\in\mathcal P$.  This completes the proof.
\end{proof}

For $\nu\in\M(N,R)$, let
\[\mathsf B_N(\nu)\coloneqq(\supp\nu,\{\pr_1,\ldots,\pr_N\},\nu),\qquad q_{N,R}(\nu)\coloneqq[\mathsf B_N(\nu)].\]
Thus $q_{N,R}\colon\M(N,R)\to\DM(N,R)$.  Write $\DM(\mathcal E;N,R)$ for the feature measurement of a family of gd-sets
in the sense of \cite[Definition~4.3]{gds2}.
The standing assumption $\TB\subset\L$ implies that $\L$ contains every translation $x\mapsto x+c$.
By \cite[Theorem~3.19]{gds2}, every $\L$-closed gd-set is then
$\L$-compact, so pyramids in $\L\circ\D$ are the $\L$-pyramids in
$\D/\L$ to which the results of \cite{gds2} apply.

\begin{lemma}
\label{lem:pyramid-measurement-compactness}
Let $\mathcal P$ be a pyramid in $\L\circ\D$.  For every positive integer
$N$ and every $R>0$,
\[\M(\mathcal P;N,R)=q_{N,R}^{-1}(\DM(\mathcal P;N,R)),\]
and this set is compact with respect to $\dpr$.
\end{lemma}

\begin{proof}
For $\nu\in\M(N,R)$, the feature family of $q_{N,R}(\nu)$ consists of the
coordinate projections.  Its measure is $\nu$ on $\supp\nu$.  The map
$q_{N,R}$ is continuous.  Indeed, if $\dpr(\mu,\nu)<\varepsilon$, Strassen's
theorem \cite[Theorem~1.22]{shioya2016mmg} gives a subtransport plan of
deficiency at most $\varepsilon$, supported where the $\ell^\infty$-distance
is at most $\varepsilon$.  Extend it to a coupling of $\mu$ and $\nu$.  On
the support of the subtransport plan, corresponding coordinate projections
differ by at most $\varepsilon$, and therefore
\[\Box(q_{N,R}(\mu),q_{N,R}(\nu))\leq2\varepsilon.\]

The forward inclusion in the asserted identity follows from the definitions.
Conversely, suppose that $q_{N,R}(\nu)\in\DM(\mathcal P;N,R)$.  Its
$\L$-saturation is dominated by an object of $\mathcal P$, and hence belongs
to $\mathcal P$.  Since $\L$ contains the identity, its coordinate functions
realize $\nu$ as an ordered measurement.  This proves the reverse inclusion.
The same argument shows that the $\L$-saturation of every
$B\in\DM(\mathcal P;N,R)$ belongs to $\mathcal P$.

The set $\DM(\mathcal P;N,R)$ is compact, and therefore closed, by \cite[Lemma~7.1]{gds2}.  The space $\M(N,R)$ is compact by
\cite[Lemma~1.17(3) and Definition~5.37]{shioya2016mmg}.
Continuity of $q_{N,R}$ and the preimage identity now give the required compactness.  This completes the proof.
\end{proof}

\begin{lemma}
\label{lem:pointwise-closure-ordered-measurement}
Let $\mathcal P$ be a pyramid in $\L\circ\D$, let $X\in\mathcal P$, and
let $g_1,\ldots,g_N\in\overline{F_X}$.  For every $R>0$,
\[ (b_R\circ g_1,\ldots,b_R\circ g_N)_*\mu_X \in\M(\mathcal P;N,R). \]
\end{lemma}

\begin{proof}
Choose a countable dense set $\{x_k\}_{k=1}^{\infty}$ in $X$.  For every
$j$ and $n$, the definition of the pointwise closure gives
$f_{j,n}\in F_X$ such that
\[
\max_{1\leq k\leq n}|f_{j,n}(x_k)-g_j(x_k)|<\frac1n.
\]
Both $f_{j,n}$ and $g_j$ are $1$-Lipschitz with respect to $d_{F_X}$.
For $x\in X$ and a fixed $k$, whenever $n\geq k$,
\[ |f_{j,n}(x)-g_j(x)| \leq2d_{F_X}(x,x_k)+\frac1n. \]
Density therefore gives $f_{j,n}(x)\to g_j(x)$ for every $x\in X$.
The approximants need not take values in $[-R,R]$, but
\Cref{eq:ordered-measurement} clips its arguments, so it is the tuples
$f_{j,n}\in F_X$ themselves that are required.  The resulting
clipped vector-valued maps converge pointwise, and hence in
$\mu_X$-measure, to $(b_R\circ g_1,\ldots,b_R\circ g_N)$.  Their
common-source couplings then give
\[
\dpr\bigl((b_R\circ f_{1,n},\ldots,b_R\circ f_{N,n})_*\mu_X,
(b_R\circ g_1,\ldots,b_R\circ g_N)_*\mu_X\bigr)\longrightarrow0.
\]
Every measure in the first argument belongs to
$\M(X;N,R)\subset\M(\mathcal P;N,R)$.  The compactness, and hence
closedness, in \Cref{lem:pyramid-measurement-compactness} places the limit
in $\M(\mathcal P;N,R)$.  This completes the proof.
\end{proof}

\begin{lemma}
\label{lem:pyramid-measurement-convergence}
For pyramids $\mathcal P_n,\mathcal P$ in $\L\circ\D$, weak convergence is
equivalent to
\[\haus{\dpr}(\M(\mathcal P_n;N,R),\M(\mathcal P;N,R))\longrightarrow0\]
for every positive integer $N$ and every $R>0$.  Moreover, if an
$\L$-gd-set $X$ satisfies
\[\M(X;N,R)\subset\M(\mathcal P;N,R)\]
for all positive integers $N$ and all $R>0$, then $X\in\mathcal P$.
\end{lemma}

\begin{proof}
Fix $N$ and $R$.  For pyramids $\mathcal P,\mathcal Q$, set
\[
\begin{aligned}
a_{N,R}&\coloneqq\haus{\dpr}(\M(\mathcal P;N,R),\M(\mathcal Q;N,R)),\\
b_{N,R}&\coloneqq\haus{\Box}(\DM(\mathcal P;N,R),\DM(\mathcal Q;N,R)).
\end{aligned}
\]
For every $B\in\DM(\mathcal P;N,R)$, enumerate its features as
$f_1,\ldots,f_N$, repeating features if there are fewer than $N$.  Then
\[
\mu\coloneqq(f_1,\ldots,f_N)_*\mu_B
\]
satisfies
\[
B=q_{N,R}(\mu),
\]
because $(f_1,\ldots,f_N)\colon B\to\supp\mu$ carries $\mu_B$ to $\mu$ and
composes with the coordinate projections to give back $f_1,\ldots,f_N$.
Therefore
$\mu\in\M(\mathcal P;N,R)$ by
\Cref{lem:pyramid-measurement-compactness}.  Continuity of $q_{N,R}$ gives
$b_{N,R}\leq2a_{N,R}$.

Take $\mu\in\M(\mathcal P;N,R)$ and $\eta>0$.  Choose
$B\in\DM(\mathcal Q;N,R)$ with
\[\Box(q_{N,R}(\mu),B)\leq b_{N,R}+\eta.\]
The measure $\mu$ is an ordered measurement of $q_{N,R}(\mu)$.
\Cref{prop:unilateral-box-comparison,prop:unilateral-box-measurement-inclusion}
therefore give a measurement of $B$ within $b_{N,R}+\eta$ of $\mu$.
The $\L$-saturation of $B$ belongs to $\mathcal Q$ by
\Cref{lem:pyramid-measurement-compactness}, so this measure lies in
$\M(\mathcal Q;N,R)$.  Taking the supremum over $\mu$, interchanging
$\mathcal P$ and $\mathcal Q$, and letting $\eta\downarrow0$ give
\[\frac12b_{N,R}\leq a_{N,R}\leq b_{N,R}.\]
The equivalence now follows from the unordered finite-measurement criterion
\cite[Proposition~7.2]{gds2}.

For the final assertion, \cite[Lemmas~4.8 and~4.9]{gds2} provide
$\L$-saturations of bounded finite-feature models converging to $X$ in box
distance.  Each such model is $q_{N,R}(\mu)$ for a measure
$\mu\in\M(X;N,R)$, with repeated coordinates added if necessary.  The
assumed inclusion and \Cref{lem:pyramid-measurement-compactness} place its
$\L$-saturation in $\mathcal P$.  Box closedness of $\mathcal P$ then gives
$X\in\mathcal P$.  This completes the proof.
\end{proof}

\section{Directional invariants}
\label{sec:directional-invariants}

Bounded finite measurements and Prokhorov control in this section yield the
limit formulas for observable diameter and multi-directed separation.

\subsection{Observable diameter}
\label{sec:observable-diameter}

A nondecreasing $1$-Lipschitz range-control map preserves the relevant
partial diameter while placing the observable in a bounded interval.

\begin{definition}[Partial diameter]
\label{def:real-line-feature-invariants}
For a Borel probability measure $\nu$ on $\R$ and $\alpha\in(0,1)$, set
\[
\pd(\nu;\alpha)\coloneqq
\inf\{\operatorname{diam}B\mid B\subset\R\text{ is Borel and }
\nu(B)\geq\alpha\}.
\]
This series writes this partial diameter as
$\operatorname{PartDiam}(\nu;\alpha)$.  It is written
$\operatorname{diam}(\nu;\alpha)$ in \cite{shioya2016mmg}.
\end{definition}

\begin{definition}[Observable diameter]
\label{def:gds-feature-invariants}
\label{def:gds-observable-diameter}
For a gd-set $X$ and $\kappa\in(0,1)$, define
\[ \od(X;-\kappa)\coloneqq \sup_{f\in\overline{F_X}}\pd(f_*\mu_X;1-\kappa). \]
For a nonempty family $\mathcal E\subset\Lup\circ\D$, set
\[ \od(\mathcal E;-\kappa)\coloneqq \sup_{X\in\mathcal E}\od(X;-\kappa). \]
\end{definition}

\begin{lemma}[Range control by a monotone map]
\label{lem:monotone-real-range-control}
For every Borel probability measure $\nu$ on $\R$, $\alpha\in(0,1)$, and
$R>0$, there is a nondecreasing $1$-Lipschitz map
\[
p\colon\R\longrightarrow[-R/\alpha,R/\alpha]
\]
such that
\[
\pd(p_*\nu;\alpha)=\min\{R,\pd(\nu;\alpha)\}.
\]
\end{lemma}

\begin{proof}
Put $r\coloneqq\min\{R,\pd(\nu;\alpha)\}$.  Tightness makes
$\pd(\nu;\alpha)$ finite.  If $r=0$, take $p=0$.  Suppose that $r>0$, and
set
\[
h(t)\coloneqq\frac{r}{\pd(\nu;\alpha)}t,\qquad \nu_h\coloneqq h_*\nu.
\]
Then $h$ is nondecreasing and $1$-Lipschitz, and
$\pd(\nu_h;\alpha)=r$.  Put
\[ x_\infty\coloneqq \inf\{x\in\R\mid\nu_h((x,+\infty))<\alpha\}. \]
Then $x_\infty\in\R$ and
$\nu_h([x_\infty,+\infty))\geq\alpha$.  Set $x_0=-\infty$ and, while
$x_j<x_\infty$, define
\[ x'_{j+1}\coloneqq \sup\{x\in\R\mid\nu_h((x_j,x))<\alpha\}. \]
If $x'_{j+1}<+\infty$, put $x_{j+1}=x'_{j+1}$.  Otherwise put
$x_{j+1}=x_\infty$ and stop.  In the former case,
$\nu_h((x_j,x_{j+1}])\geq\alpha$ and $x_{j+1}-x_j\geq r$ for $j\geq1$.
In the latter case,
$\nu_h([x_{j+1},+\infty))\geq\alpha$.  Hence the construction reaches
$x_\infty$ after a finite number $m$ of steps and $m\alpha\leq1$.

Set
\[
E\coloneqq\bigcup_{j=1}^m(x_j-r,x_j+r),\qquad
\varphi(t)\coloneqq-mr+\int_{-\infty}^t\boldsymbol 1_E(s)\,ds.
\]
The map $\varphi$ is nondecreasing and $1$-Lipschitz, and
\[
-\frac R\alpha\leq-mr\leq\varphi(t)\leq mr\leq\frac R\alpha.
\]
Moreover, $\varphi(x_j+a)=\varphi(x_j)+a$ for $|a|<r$.
In particular, $\varphi$ is strictly increasing on
$(x_j-r,x_j+r)$, and
\begin{align*}
\varphi^{-1}((-\infty,\varphi(x_j)+a))&=(-\infty,x_j+a),\\
\varphi^{-1}((\varphi(x_j)+a,+\infty))&=(x_j+a,+\infty)
\end{align*}
whenever $|a|<r$.

Let $c\leq d$ with $d-c<r$.  We show that
$\nu_h(\varphi^{-1}([c,d]))<\alpha$ in all four possible positions of the
interval.
\begin{enumerate}[label=\textup{(\roman*)}]
\item If $d<\varphi(x_1)$, choose $\eta\in(0,r)$ such that
$d<\varphi(x_1)-\eta$.  Then
\[ \nu_h(\varphi^{-1}([c,d])) \leq\nu_h((-\infty,x_1-\eta))<\alpha. \]
\item If $c\leq\varphi(x_j)\leq d$, put
$a\coloneqq\varphi(x_j)-c$ and $b\coloneqq d-\varphi(x_j)$.  Then
$a,b\geq0$, $a+b<r$, and
\[
\varphi^{-1}([c,d])=[x_j-a,x_j+b].
\]
This interval has diameter less than $r$, so its measure is less than
$\alpha$.
\item If $\varphi(x_j)<c\leq d<\varphi(x_{j+1})$ for $j<m$, choose
$\eta\in(0,r)$ such that $d<\varphi(x_{j+1})-\eta$.  Then
\[ \nu_h(\varphi^{-1}([c,d])) \leq\nu_h((x_j,x_{j+1}-\eta))<\alpha. \]
\item If $\varphi(x_m)<c$, choose $\eta\in(0,r)$ such that
$\varphi(x_m)+\eta<c$.  Then
\[
\nu_h(\varphi^{-1}([c,d]))
\leq\nu_h((x_m+\eta,+\infty))
\leq\nu_h((x_\infty+\eta,+\infty))<\alpha.
\]
\end{enumerate}
Thus no Borel set of diameter less than $r$ has
$\varphi_*\nu_h$-measure at least $\alpha$.
Hence $\pd(\varphi_*\nu_h;\alpha)\geq r$, while the reverse inequality
follows from the $1$-Lipschitz property.  The map $p\coloneqq\varphi\circ h$
has all the required properties.  This completes the proof.
\end{proof}

\begin{proposition}[Observable diameter from one measurement]
\label{prop:observable-diameter-from-one-measurement}
Let $\mathcal E\subset\Lup\circ\D$ be nonempty.  For
$\kappa\in(0,1)$ and $R>0$, set
\[ \mathcal O(\mathcal E;\kappa,R)\coloneqq \overline{\M(\mathcal E;1,R/(1-\kappa))}^{\dpr}. \]
Then
\begin{equation}
\label{eq:observable-diameter-from-one-measurement}
\min\{R,\od(\mathcal E;-\kappa)\}
=\sup_{\nu\in\mathcal O(\mathcal E;\kappa,R)}
\min\{R,\pd(\nu;1-\kappa)\}.
\end{equation}
\end{proposition}

\begin{proof}
Every measure in $\M(\mathcal E;1,R/(1-\kappa))$ has partial diameter at
mass $1-\kappa$ at most $\od(\mathcal E;-\kappa)$.  We use the known
lower-semicontinuity fact that if Borel probability measures $\nu_n$ on
$\R$ converge weakly to $\nu$, then
\[ \pd(\nu;\alpha)\leq\liminf_{n\to\infty}\pd(\nu_n;\alpha) \qquad(\alpha\in(0,1)); \]
see \cite[Lemmas~3.8 and~3.1(1)]{ozawa2015limit}.  It follows that the
same bound holds on the $\dpr$-closure.  Conversely, apply
\Cref{lem:monotone-real-range-control} to $f_*\mu_X$ for
$X\in\mathcal E$ and $f\in\overline{F_X}$.  The resulting
$p\in\Lup$ has range in
$[-R/(1-\kappa),R/(1-\kappa)]$, and $(p\circ f)_*\mu_X$ realizes the
truncated partial diameter.  Taking the supremum over $X$ and $f$ proves
the formula.  This completes the proof.
\end{proof}

\begin{proposition}[Observable-diameter limit for $\Lup$-pyramids]
\label{prop:lup-observable-diameter-pyramid-limit}
Let $\mathcal P_n\to\mathcal P$ weakly be pyramids in $\Lup\circ\D$.  For
every $\kappa\in(0,1)$,
\begin{align}
\od(\mathcal P;-\kappa)
&=\lim_{\varepsilon\downarrow0}\liminf_{n\to\infty}
\od(\mathcal P_n;-(\kappa+\varepsilon)),
\label{eq:obsdiam-pyramid-limit-liminf}\\
&=\lim_{\varepsilon\downarrow0}\limsup_{n\to\infty}
\od(\mathcal P_n;-(\kappa+\varepsilon))
\label{eq:obsdiam-pyramid-limit-limsup}
\end{align}
in $[0,+\infty]$, where $0<\varepsilon<1-\kappa$.
\end{proposition}

\begin{proof}
The partial diameter is right-continuous in $\kappa$
\cite[Lemma~4.13]{gds1}.  For every real number
$a<\od(\mathcal P;-\kappa)$, choose
$X\in\mathcal P$ and $f\in\overline{F_X}$ such that
\[
\pd(f_*\mu_X;1-\kappa)>a.
\]
Applying right-continuity to $f_*\mu_X$ and then letting
$a\uparrow\od(\mathcal P;-\kappa)$ shows that
\begin{equation}
\label{eq:pyramid-obsdiam-right-continuity}
\od(\mathcal P;-\kappa)=
\lim_{\varepsilon\downarrow0}\od(\mathcal P;-(\kappa+\varepsilon)).
\end{equation}

Fix $0<2\varepsilon<1-\kappa$ and $R>0$.  Choose
$\nu\in\mathcal O(\mathcal P;\kappa+2\varepsilon,R)$ approaching the
supremum in \Cref{eq:observable-diameter-from-one-measurement}.
\Cref{lem:pyramid-measurement-convergence} gives
$\nu_n\in\mathcal O(\mathcal P_n;\kappa+2\varepsilon,R)$ with
$\dpr(\nu_n,\nu)\to0$.  We use the estimate
\[
\pd(\mu;\beta)\leq\pd(\nu;\beta+\delta)+2\delta
\]
for Borel probability measures $\mu,\nu$ on $\R$ whenever
$\beta,\delta>0$, $\beta+\delta<1$, and $\dpr(\mu,\nu)<\delta$
\cite[Lemma~4.11]{gds1}.  If
$\dpr(\nu_n,\nu)<\delta_n<\varepsilon$ and $\delta_n\downarrow0$, then
\begin{align*}
\pd(\nu;1-\kappa-2\varepsilon) &\leq\pd(\nu_n;1-\kappa-2\varepsilon+\delta_n)+2\delta_n\\
&\leq\od(\mathcal P_n;-(\kappa+\varepsilon))+2\delta_n.
\end{align*}
After truncating the left-hand side at $R$ and passing to $n\to\infty$,
take the supremum over
$\nu\in\mathcal O(\mathcal P;\kappa+2\varepsilon,R)$ and apply
\Cref{eq:observable-diameter-from-one-measurement}.  Letting
$R\to+\infty$ then yields
\begin{equation}
\label{eq:obsdiam-limit-lower-squeeze}
\od(\mathcal P;-(\kappa+2\varepsilon))
\leq\liminf_{n\to\infty}\od(\mathcal P_n;-(\kappa+\varepsilon)).
\end{equation}

Conversely, choose
$\nu_n\in\mathcal O(\mathcal P_n;\kappa+\varepsilon,R)$ within $1/n$ of the
truncated supremum.  \Cref{lem:pyramid-measurement-convergence} gives
$\nu'_n\in\mathcal O(\mathcal P;\kappa+\varepsilon,R)$ with
$\dpr(\nu_n,\nu'_n)\to0$.  Taking
$\dpr(\nu_n,\nu'_n)<\delta_n<\varepsilon/2$ and $\delta_n\to0$ gives
\begin{align*}
\pd(\nu_n;1-\kappa-\varepsilon) &\leq\pd(\nu'_n;1-\kappa-\varepsilon+\delta_n)+2\delta_n\\
&\leq\od(\mathcal P;-(\kappa+\varepsilon/2))+2\delta_n.
\end{align*}
Letting $R$ increase if necessary gives
\begin{equation}
\label{eq:obsdiam-limit-upper-squeeze}
\limsup_{n\to\infty}\od(\mathcal P_n;-(\kappa+\varepsilon))
\leq\od(\mathcal P;-(\kappa+\varepsilon/2)).
\end{equation}
Letting $\varepsilon\downarrow0$ in the two squeeze estimates and using
\Cref{eq:pyramid-obsdiam-right-continuity} proves the result.  This
completes the proof.
\end{proof}

\subsection{Multi-directed separation}
\label{sec:multi-directed-separation}

Coordinates indexed by pairs of witness sets encode every ordered gap in one
finite measurement.

Let $S$ be a set with a Borel probability measure $\mu_S$, let $N$ be a positive integer, and let
$\kappa_0,\ldots,\kappa_N\in(0,1)$.  Write
$\mathcal A_S(\kappa_0,\ldots,\kappa_N)$ for the tuples
$(A_0,\ldots,A_N)$ of pairwise disjoint Borel subsets of $S$ such that
$\mu_S(A_i)\geq\kappa_i$ for every $i$.

For a gd-set $X$ and nonempty $A,B\subset X$, set
\[ \sigma^{\to}(X;A,B)\coloneqq \sup_{f\in\overline{F_X}}\inf_{a\in A,\,b\in B}(f(b)-f(a)). \]
If the zero function belongs to $\overline{F_X}$, then
$\sigma^{\to}(X;A,B)\geq0$.  This holds for $\mathcal L$-gd-sets with
$\TB\subset\mathcal L$, the setting of every subsequent result using this
quantity.

\begin{definition}[Multi-directed separation]
\label{def:multi-directed-separation}
\label{def:directional-invariants}
For a gd-set $X$, set
\[
\operatorname{Sep}(X;\kappa_0,\ldots,\kappa_N)
\coloneqq
\sup_{\substack{(A_0,\ldots,A_N)\in\\
\mathcal A_X(\kappa_0,\ldots,\kappa_N)}}
\min_{0\leq i<j\leq N}\sigma^{\to}(X;A_i,A_j).
\]
If the admissible family is empty, the separation distance is $0$.
\end{definition}

For a nonempty family $\mathcal E$ of gd-sets, set
\[
\operatorname{Sep}(\mathcal E;\kappa_0,\ldots,\kappa_N)
\coloneqq\sup_{X\in\mathcal E}
\operatorname{Sep}(X;\kappa_0,\ldots,\kappa_N).
\]
In particular, this defines the separation distance of a pyramid.

Put $m_N\coloneqq N(N+1)/2$ and index $\R^{m_N}$ by pairs $(i,j)$ with
$0\leq i<j\leq N$.  For $r>0$ and $R\geq r$, set
\begin{align*}
E_k(r;R)&\coloneqq \{z=(z_{ij})_{i<j}\in[-R,R]^{m_N}\mid z_{kj}=0\ (k<j),\ z_{ik}=r\ (i<k)\},\\
E_k(r)&\coloneqq E_k(r;r).
\end{align*}

\begin{lemma}[Measurement representation of multi-separation]
\label{lem:tb-pyramid-multi-separation-measurement}
Let $\mathcal P$ be a pyramid in $\TB\circ\D$.  For every positive integer
$N$ and $\kappa_0,\ldots,\kappa_N\in(0,1)$,
\[
\operatorname{Sep}(\mathcal P;\kappa_0,\ldots,\kappa_N)
=\sup\left\{r>0\ \middle|\
\begin{array}{l}
\text{there is }\nu\in\M(\mathcal P;m_N,r)\text{ such that}\\
\nu(E_k(r))\geq\kappa_k\quad(0\leq k\leq N)
\end{array}\right\}.
\]
If the set on the right is empty, its value is $0$.
\end{lemma}

\begin{proof}
Take $0<r<\operatorname{Sep}(\mathcal P;\kappa_0,\ldots,\kappa_N)$.
There are $Z\in\mathcal P$ and
$(A_0,\ldots,A_N)\in\mathcal A_Z(\kappa_0,\ldots,\kappa_N)$ such that
\[
r<\min_{0\leq i<j\leq N}\sigma^{\to}(Z;A_i,A_j).
\]
For each $i<j$, choose $f_{ij}\in\overline{F_Z}$ with
\[
\inf_{a\in A_i,\,b\in A_j}(f_{ij}(b)-f_{ij}(a))>r,
\]
and set
\[ a_{ij}\coloneqq\sup_{a\in A_i}f_{ij}(a),\qquad h_{ij}(t)\coloneqq\min\{r,\max\{0,t-a_{ij}\}\}. \]
Then $h_{ij}\in\TB$ and $h_{ij}\circ f_{ij}\in\overline{F_Z}$.
\Cref{lem:pointwise-closure-ordered-measurement} gives
\[ \nu\coloneqq((h_{ij}\circ f_{ij})_{i<j})_*\mu_Z \in\M(\mathcal P;m_N,r). \]
The coordinates with source $k$ vanish on $A_k$, while those with target
$k$ equal $r$ there.  Hence $\nu(E_k(r))\geq\kappa_k$.

Conversely, let $\nu$ satisfy the displayed condition.  For some
$Y\in\mathcal P$ and $f_{ij}\in F_Y$, put
$H\coloneqq(b_r\circ f_{ij})_{i<j}$, so that $\nu=H_*\mu_Y$.  The sets
$B_k\coloneqq H^{-1}(E_k(r))$ are pairwise disjoint and have measure at
least $\kappa_k$.  If $i<j$, $x\in B_i$, and $y\in B_j$, then
$f_{ij}(x)=0$ and $f_{ij}(y)\geq r$.  Thus
$\sigma^{\to}(Y;B_i,B_j)\geq r$.  This completes the proof.
\end{proof}

\begin{lemma}
\label{lem:multi-directed-separation-prokhorov-control}
Let $Y$ be a gd-set, $r>0$, $R\geq r$, and $\nu\in\M(m_N,R)$.  Suppose
$\nu(E_k(r;R))\geq\kappa_k$ for every $k$ and
\[
0<\delta<\min\{r/2,\kappa_0,\ldots,\kappa_N\}.
\]
If $\nu'\in\M(Y;m_N,R)$ and $\dpr(\nu,\nu')<\delta$, then
\[ \operatorname{Sep}(Y;\kappa_0-\delta,\ldots,\kappa_N-\delta) \geq r-2\delta. \]
\end{lemma}

\begin{proof}
Choose $f_{ij}\in F_Y$ representing $\nu'$ and put
$H\coloneqq(b_R\circ f_{ij})_{i<j}$.  Define
\[ B_k\coloneqq\{y\in Y\mid b_R(f_{kj}(y))<\delta\ (k<j),\quad b_R(f_{ik}(y))>r-\delta\ (i<k)\}. \]
Choose $\delta''$ such that
$\dpr(\nu,\nu')<\delta''<\delta$.  By the definition of the closed
neighborhood,
\[
H^{-1}\bigl(\mathrm B(E_k(r;R),\delta'';\|\cdot\|_\infty)\bigr)\subset B_k.
\]
Indeed, the coordinates with source $k$ are at most
$\delta''<\delta$, while those with target $k$ are at least
$r-\delta''>r-\delta$.  The Prokhorov inequality now gives
\[
\mu_Y(B_k)\geq\nu'(\mathrm B(E_k(r;R),\delta'';\|\cdot\|_\infty))
\geq\nu(E_k(r;R))-\delta''\geq\kappa_k-\delta.
\]
Since $r>2\delta$, the sets $B_k$ are pairwise disjoint.  If $i<j$,
$x\in B_i$, and $y\in B_j$, then
$f_{ij}(y)-f_{ij}(x)>r-2\delta$.  Consequently
$\sigma^{\to}(Y;B_i,B_j)\geq r-2\delta$, proving the assertion.  This
completes the proof.
\end{proof}

\begin{proposition}[Multi-separation limit for $\TB$-pyramids]
\label{prop:tb-multi-directed-separation-pyramid-limit}
Let $\mathcal P_n\to\mathcal P$ weakly be pyramids in $\TB\circ\D$.  For
every positive integer $N$ and $\kappa_0,\ldots,\kappa_N\in(0,1)$,
\begin{align}
\operatorname{Sep}(\mathcal P;\kappa_0,\ldots,\kappa_N)
&=\lim_{\varepsilon\downarrow0}\liminf_{n\to\infty}
\operatorname{Sep}(\mathcal P_n;\kappa_0-\varepsilon,\ldots,
\kappa_N-\varepsilon),
\label{eq:multi-directed-separation-pyramid-limit-liminf}\\
&=\lim_{\varepsilon\downarrow0}\limsup_{n\to\infty}
\operatorname{Sep}(\mathcal P_n;\kappa_0-\varepsilon,\ldots,
\kappa_N-\varepsilon)
\label{eq:multi-directed-separation-pyramid-limit-limsup}
\end{align}
in $[0,+\infty]$, where
$0<\varepsilon<\min_{0\leq j\leq N}\kappa_j$.
\end{proposition}

\begin{proof}
Let $L$ denote the limit on the right-hand side of the statement. Monotonicity in the
mass parameters gives
\[
\operatorname{Sep}(\mathcal P;\kappa_0,\ldots,\kappa_N)\leq L.
\]
If $L=0$, nonnegativity gives equality. Suppose now that $L>0$, and fix a
finite number $0<r<L$ to prove the reverse inequality. There are numbers
$\varepsilon_\ell\downarrow0$ such that
\[ r<\operatorname{Sep}(\mathcal P;\kappa_0-\varepsilon_\ell,\ldots, \kappa_N-\varepsilon_\ell). \]
By \Cref{lem:tb-pyramid-multi-separation-measurement}, choose
$\nu_\ell\in\M(\mathcal P;m_N,r)$ satisfying
\[ \nu_\ell(E_k(r))\geq\kappa_k-\varepsilon_\ell \qquad(0\leq k\leq N). \]
\Cref{lem:pyramid-measurement-compactness} gives a subsequence converging weakly
to some $\nu\in\M(\mathcal P;m_N,r)$. Each $E_k(r)$ is closed, so the
Portmanteau theorem yields
\[
\nu(E_k(r))\geq\limsup_{\ell\to\infty}\nu_\ell(E_k(r))\geq\kappa_k.
\]
Another application of
\Cref{lem:tb-pyramid-multi-separation-measurement} gives
$r\leq\operatorname{Sep}(\mathcal P;\kappa_0,\ldots,\kappa_N)$. Letting
$r\uparrow L$ when $L<+\infty$, or taking $r$ arbitrarily large when
$L=+\infty$, proves the reverse inequality. Therefore,
\begin{equation}
\label{eq:pyramid-multi-directed-separation-right-continuity}
\operatorname{Sep}(\mathcal P;\kappa_0,\ldots,\kappa_N)
=\lim_{\varepsilon\downarrow0}
\operatorname{Sep}(\mathcal P;\kappa_0-\varepsilon,\ldots,
\kappa_N-\varepsilon).
\end{equation}

Fix $0<\varepsilon<\min_j\kappa_j$ and choose $r$ such that
\[
0<r<\operatorname{Sep}(\mathcal P;\kappa_0,\ldots,\kappa_N).
\]
Represent $r$ by $\nu\in\M(\mathcal P;m_N,r)$.
\Cref{lem:pyramid-measurement-convergence} gives
$\nu_n\in\M(\mathcal P_n;m_N,r)$ with $\dpr(\nu_n,\nu)\to0$.  Applying
\Cref{lem:multi-directed-separation-prokhorov-control} with errors
$\delta_n\downarrow0$ smaller than $\min\{r/2,\varepsilon\}$ yields
\begin{equation}
\label{eq:multi-directed-separation-limit-lower-squeeze}
\operatorname{Sep}(\mathcal P;\kappa_0,\ldots,\kappa_N)
\leq\liminf_{n\to\infty}
\operatorname{Sep}(\mathcal P_n;\kappa_0-\varepsilon,\ldots,
\kappa_N-\varepsilon).
\end{equation}

For the reverse estimate, take $0<2\varepsilon<\min_j\kappa_j$.  Suppose that
\[
\limsup_{n\to\infty}
\operatorname{Sep}(\mathcal P_n;\kappa_0-\varepsilon,\ldots,
\kappa_N-\varepsilon)
>\operatorname{Sep}(\mathcal P;\kappa_0-2\varepsilon,\ldots,
\kappa_N-2\varepsilon).
\]
Choose a finite $r$ strictly between the two sides.  After passing to a
subsequence, \Cref{lem:tb-pyramid-multi-separation-measurement} gives
witnesses satisfying
\[\nu_n\in\M(\mathcal P_n;m_N,r).\]
\Cref{lem:pyramid-measurement-convergence} gives
$\nu'_n\in\M(\mathcal P;m_N,r)$ with $\dpr(\nu_n,\nu'_n)\to0$.  Apply
\Cref{lem:multi-directed-separation-prokhorov-control} with
$\dpr(\nu_n,\nu'_n)<\delta_n<\min\{r/2,\varepsilon\}$ and
$\delta_n\to0$.  It follows that
\[
\operatorname{Sep}(\mathcal P;\kappa_0-2\varepsilon,\ldots,
\kappa_N-2\varepsilon)\geq r-2\delta_n.
\]
Letting $n\to\infty$ contradicts the choice of $r$.  Hence
\begin{equation}
\label{eq:multi-directed-separation-limit-upper-squeeze}
\limsup_{n\to\infty}
\operatorname{Sep}(\mathcal P_n;\kappa_0-\varepsilon,\ldots,
\kappa_N-\varepsilon)
\leq\operatorname{Sep}(\mathcal P;\kappa_0-2\varepsilon,\ldots,
\kappa_N-2\varepsilon).
\end{equation}
Letting $\varepsilon\downarrow0$ in the two squeeze estimates and using
\Cref{eq:pyramid-multi-directed-separation-right-continuity} proves the
formula.  This completes the proof.
\end{proof}

\section{Elementary comparisons of concentration for gd-sets}
\label{sec:one-sided-concentration}

When applying the preceding limit formula for multi-directed separation to
comparisons of concentration quantities, one must account for the fact that
directed gaps can be negative for general gd-sets.  Indeed,
$\sigma^\to$ need not be nonnegative when $\overline{F_X}$ contains no
constant functions.  We therefore take the nonnegative part of the
separation distance.

\begin{definition}[Nonnegative directed separation]
\label{def:nonnegative-multi-directed-separation}
For a gd-set $X$, a positive integer $N$, and
$\kappa_0,\ldots,\kappa_N\in(0,1)$, define
\[
\operatorname{Sep}^+(X;\kappa_0,\ldots,\kappa_N)
\coloneqq\max\{0,\operatorname{Sep}(X;\kappa_0,\ldots,\kappa_N)\}.
\]
\end{definition}

This normalization has genuine content for general gd-sets.  Consider
$X=\{x_-,x_0,x_+\}$ with
\[
\begin{aligned}
\mu_X&=\frac15\delta_{x_-}+\frac35\delta_{x_0}+\frac15\delta_{x_+},
\qquad F_X=\{f\},\\
f(x_-)&=0,\quad f(x_0)=1,\quad f(x_+)=2.
\end{aligned}
\]
Since $f$ is injective, $d_{F_X}$ is a metric.  For $\kappa=3/10$, two
disjoint sets can both have measure at least $\kappa$ only if one is
$\{x_0\}$ and the other is $\{x_-,x_+\}$.  Their directed gap is $-1$ in
either order, and therefore
$\operatorname{Sep}(X;3/10,3/10)=-1$.

Since $\overline{F_X}$ need not be closed under negation, we retain the upper
and lower tails separately.

\begin{definition}[Medians and one-sided concentration functions]
\label{def:one-sided-feature-concentration}
For a gd-set $Z$ and $f\in\overline{F_Z}$, set
\[
\operatorname{Med}(f)\coloneqq
\{m\in\R\mid\mu_Z(f\leq m)\geq1/2,\ \mu_Z(f\geq m)\geq1/2\}.
\]
For $r>0$, define
\begin{align*}
\alpha_Z^+(r)&\coloneqq\sup_{f\in\overline{F_Z}}
\sup_{m\in\operatorname{Med}(f)}\mu_Z(f\geq m+r),\\
\alpha_Z^-(r)&\coloneqq\sup_{f\in\overline{F_Z}}
\sup_{m\in\operatorname{Med}(f)}\mu_Z(f\leq m-r).
\end{align*}
\end{definition}

Under the common sign convention $f(y)-f(x)\leq d_X(x,y)$, the functions
$\alpha_Z^+$ and $\alpha_Z^-$ for $Z=\operatorname{Rep}^{+}(X)$ are
precisely the right and left concentration functions $\alpha^R$ and
$\alpha^L$, respectively, of
\cite[Definitions~3.2 and~3.4 and Lemma~3.6]{stojmirovic2004quasi}.
See also \cite{ohta2021comparison} for the curvature-to-concentration
direction on irreversible Finsler manifolds, and
\cite{cheng-feng2025concentration} for a different min-combined
one-variable concentration function.

The set of medians need not be a singleton.  For example, if
$f_*\mu_Z=(\delta_0+\delta_1)/2$, then
$\operatorname{Med}(f)=[0,1]$.  This is why
\Cref{def:one-sided-feature-concentration} takes the supremum over all
medians.

\begin{proposition}[One-sided concentration formulas]
\label{prop:directional-invariant-asymmetric-specialization}
For a qm-space $X$ and $r>0$, let $\mathcal H_X$ be the family of Borel
subsets of measure at least $1/2$, and set
\[
U_r^+(A)\coloneqq\{x\mid\inf_{a\in A}d_X(a,x)<r\},\qquad
U_r^-(A)\coloneqq\{x\mid\inf_{a\in A}d_X(x,a)<r\}.
\]
Then
\begin{align*}
\alpha_{\operatorname{Rep}^{+}(X)}^+(r)
&=\sup_{A\in\mathcal H_X}\{1-\mu_X(U_r^+(A))\},\\
\alpha_{\operatorname{Rep}^{+}(X)}^-(r)
&=\sup_{A\in\mathcal H_X}\{1-\mu_X(U_r^-(A))\}.
\end{align*}
If $X$ is an mm-space, then
\[
U_r(A)\coloneqq\{x\mid\inf_{a\in A}d_X(a,x)<r\}
\]
satisfies $U_r^+(A)=U_r^-(A)=U_r(A)$, and the two formulas recover the
classical concentration function.
\end{proposition}

\begin{proof}
The one-sided Lipschitz inequalities give one inclusion for each median
tail.  The functions
$x\mapsto\inf_{a\in A}d_X(a,x)$ and
$x\mapsto-\inf_{a\in A}d_X(x,a)$ give the reverse inclusions.  The final
assertion follows from the symmetry of the distance.  This completes the
proof.
\end{proof}

Since distance functions from sets need not belong to $\overline{F_X}$, the
next comparison uses the functions that define the gaps and tails themselves
as witnesses.

\begin{theorem}[Elementary comparisons for function-family concentration]
\label{thm:elementary-feature-concentration-comparison}
Let $X$ be a gd-set.  Then the following hold.
\begin{enumerate}[label=\textup{(\roman*)}]
\item If $0<\kappa<1/2$, then
\[
\od(X;-2\kappa)\leq\operatorname{Sep}^+(X;\kappa,\kappa).
\]
\item If $0<\kappa'<\kappa\leq1/2$, then
\[
\operatorname{Sep}^+(X;\kappa,\kappa)\leq\od(X;-\kappa').
\]
\item If $0<\kappa<1$ and $r_+,r_->0$ satisfy
\[
\alpha_X^+(r_+)+\alpha_X^-(r_-)\leq\kappa,
\]
then
\[
\od(X;-\kappa)\leq r_++r_-.
\]
\item For $r>0$,
\[
\max\{\alpha_X^+(r),\alpha_X^-(r)\}
\leq\sup\{\kappa\in(0,1)\mid\od(X;-\kappa)\geq r\}.
\]
If the set on the right is empty, its supremum is understood to be $0$.
\end{enumerate}
These comparisons require no closure of the function family under maximum or
minimum, truncation, postcomposition by $1$-Lipschitz maps, or
inf-convolution.  They also require neither the inclusion of constant
functions nor $0\in\overline{F_X}$.
\end{theorem}

\begin{proof}
For \textup{(i)}, use the lower and upper quantile sets.  Take
$f\in\overline{F_X}$ and put $\nu\coloneqq f_*\mu_X$.  Define
\[
\rho_-\coloneqq\sup\{t\in\R\mid\nu((-\infty,t))\leq\kappa\},\qquad
\rho_+\coloneqq\inf\{t\in\R\mid\nu((t,+\infty))\leq\kappa\}.
\]
The elementary properties of quantiles give
\[
\nu((-\infty,\rho_-])\geq\kappa,\quad
\nu([\rho_+,+\infty))\geq\kappa,\quad
\nu([\rho_-,\rho_+])\geq1-2\kappa,\quad
\rho_-\leq\rho_+.
\]
Therefore,
$\pd(\nu;1-2\kappa)\leq\rho_+-\rho_-$.  If
$\rho_-<\rho_+$, then
\[
A_0\coloneqq f^{-1}((-\infty,\rho_-]),\qquad
A_1\coloneqq f^{-1}([\rho_+,+\infty))
\]
are disjoint, each has measure at least $\kappa$, and
\[
\sigma^\to(X;A_0,A_1)
\geq\inf_{a\in A_0,\,b\in A_1}(f(b)-f(a))
\geq\rho_+-\rho_-.
\]
If $\rho_-=\rho_+$, the partial diameter is $0$.  Thus, in either case,
\[
\pd(f_*\mu_X;1-2\kappa)
\leq\operatorname{Sep}^+(X;\kappa,\kappa).
\]
Taking the supremum over $f$ proves \textup{(i)}.

To prove \textup{(ii)}, take disjoint Borel sets $A_0,A_1$ such that
\[
\mu_X(A_i)\geq\kappa\qquad(i=0,1).
\]
If
$\sigma^\to(X;A_0,A_1)\leq0$, the assertion follows from nonnegativity.
Otherwise, take $0<s<\sigma^\to(X;A_0,A_1)$ and choose
$f\in\overline{F_X}$ such that
\[
\inf_{a\in A_0,\,b\in A_1}(f(b)-f(a))>s.
\]
Let $C\subset\R$ be a Borel set such that
$f_*\mu_X(C)\geq1-\kappa'$.  For $i=0,1$,
\[
\mu_X(f^{-1}(C))+\mu_X(A_i)\geq1-\kappa'+\kappa>1,
\]
and hence $f^{-1}(C)\cap A_i\neq\varnothing$.  Thus $C$ meets both
$f(A_0)$ and $f(A_1)$.  Since $\operatorname{diam}C>s$, we have
\[
\pd(f_*\mu_X;1-\kappa')\geq s.
\]
Let $s\uparrow\sigma^\to(X;A_0,A_1)$ and take the supremum over
$A_0,A_1$ to obtain \textup{(ii)}.

For \textup{(iii)}, take $f\in\overline{F_X}$ and
$m\in\operatorname{Med}(f)$.  The assumption gives
\[
\mu_X(m-r_-<f<m+r_+)
\geq1-\alpha_X^-(r_-)-\alpha_X^+(r_+)\geq1-\kappa.
\]
The interval has diameter $r_++r_-$, so
$\pd(f_*\mu_X;1-\kappa)\leq r_++r_-$.  Taking the supremum over $f$ proves
the assertion.

Consider the upper tail in \textup{(iv)}.  Take
$f\in\overline{F_X}$, $m\in\operatorname{Med}(f)$, and
$0<\kappa<\mu_X(f\geq m+r)$.  The set $\{f\leq m\}$ has measure at least
$1/2$ and is disjoint from the upper tail, so $\kappa<1/2$.  Let
$C\subset\R$ be any Borel set such that
$f_*\mu_X(C)\geq1-\kappa$.  Since
$\mu_X(f^{-1}(C))\geq1-\kappa$ and
$\mu_X(f\leq m)\geq1/2>\kappa$, the set $C$ meets
$(-\infty,m]$.  It also meets $[m+r,+\infty)$ because
$\mu_X(f\geq m+r)>\kappa$.  Therefore,
$\operatorname{diam}C\geq r$, and hence $\od(X;-\kappa)\geq r$.  Letting
$\kappa\uparrow\mu_X(f\geq m+r)$ and taking the supremum over $f$ and $m$
give the estimate for $\alpha_X^+(r)$.  For $\alpha_X^-(r)$, use
$[m,+\infty)$ and $(-\infty,m-r]$.  This completes the proof.
\end{proof}

We express the concentration of each function around a suitable constant by
one metric on measurable functions.  For $\mu_X$-measurable functions
$f,g\colon X\to\R$, define their Ky Fan distance by
\[
d_{\mathrm{KF}}^{\mu_X}(f,g)\coloneqq
\inf\{\varepsilon\geq0\mid
\mu_X(\{x\in X\mid|f(x)-g(x)|>\varepsilon\})\leq\varepsilon\}.
\]
We also use $c\in\R$ for the constant function $x\mapsto c$.

\begin{definition}[Function-family Lévy family]
\label{def:feature-levy-family}
For a sequence $\{X_n\}$ of gd-sets, set
\[
q(X)\coloneqq\sup_{f\in\overline{F_X}}
\inf_{c\in\R}d_{\mathrm{KF}}^{\mu_X}(f,c).
\]
We call $\{X_n\}$ a function-family Lévy family if
$q(X_n)\to0$.
\end{definition}

This is the function-family counterpart of the classical notion of a Lévy
family of mm-spaces \cite[Definition~2.14]{shioya2016mmg}.

\begin{corollary}[Characterizations of function-family Lévy families]
\label{cor:feature-levy-characterizations}
For a sequence $\{X_n\}$ of gd-sets, the following conditions are
equivalent.
\begin{enumerate}[label=\textup{(\roman*)}]
\item The sequence $\{X_n\}$ is a function-family Lévy family.
\item For every $\kappa\in(0,1)$,
$\od(X_n;-\kappa)\to0$.
\item For every $\kappa\in(0,1)$,
$\operatorname{Sep}^+(X_n;\kappa,\kappa)\to0$.
\item For every $r>0$,
$\alpha_{X_n}^+(r)\to0$ and $\alpha_{X_n}^-(r)\to0$.
\end{enumerate}
If $0\in\overline{F_{X_n}}$ for every $n$, then
$\operatorname{Sep}^+$ in \textup{(iii)} can be replaced by
$\operatorname{Sep}$.
\end{corollary}

\begin{proof}
We first prove that \textup{(i)} implies \textup{(ii)}.  Take
$\kappa\in(0,1)$ and $0<\varepsilon<\kappa$.  For all sufficiently large
$n$, we have $q(X_n)<\varepsilon$.  For each
$f\in\overline{F_{X_n}}$, there is a $c\in\R$ such that
$d_{\mathrm{KF}}^{\mu_{X_n}}(f,c)<\varepsilon$.  Thus
\[
\mu_{X_n}(|f-c|\leq\varepsilon)
\geq1-\varepsilon\geq1-\kappa.
\]
It follows that $\od(X_n;-\kappa)\leq2\varepsilon$, which proves
\textup{(ii)}.

Conversely, assume \textup{(ii)} and take $0<\varepsilon<1$.  For all
sufficiently large $n$,
$\od(X_n;-\varepsilon/2)<\varepsilon/2$.  For each
$f\in\overline{F_{X_n}}$, there is a nonempty Borel set $B\subset\R$ of
diameter less than $\varepsilon/2$ and
$f_*\mu_{X_n}$-measure at least $1-\varepsilon/2$.  If $c\in B$, then
$d_{\mathrm{KF}}^{\mu_{X_n}}(f,c)\leq\varepsilon/2$.  This estimate is
uniform in $f$, so $q(X_n)\leq\varepsilon/2$, proving \textup{(i)}.

Parts \textup{(i)} and \textup{(ii)} of
\Cref{thm:elementary-feature-concentration-comparison} show that
\textup{(ii)} and \textup{(iii)} are equivalent.  Indeed, if
$0<\kappa\leq1/2$, then
\[
\operatorname{Sep}^+(X_n;\kappa,\kappa)
\leq\od(X_n;-\kappa/2).
\]
If $1/2<\kappa<1$, there are no admissible pairs of sets.  In the other
direction, for every $\eta\in(0,1)$ use
\[
\od(X_n;-\eta)
\leq\operatorname{Sep}^+(X_n;\eta/2,\eta/2).
\]

Suppose that \textup{(ii)} holds.  For fixed $r>0$ and
$\varepsilon\in(0,1)$, all sufficiently large $n$ satisfy
$\od(X_n;-\varepsilon)<r$.  Since
$\kappa\mapsto\od(X_n;-\kappa)$ is nonincreasing,
\Cref{thm:elementary-feature-concentration-comparison}\textup{(iv)} gives
\[
\max\{\alpha_{X_n}^+(r),\alpha_{X_n}^-(r)\}\leq\varepsilon.
\]
This proves \textup{(iv)}.  Conversely, assume \textup{(iv)} and fix
$\kappa\in(0,1)$ and $\varepsilon>0$.  For all sufficiently large $n$,
\[
\alpha_{X_n}^+(\varepsilon/2)+
\alpha_{X_n}^-(\varepsilon/2)\leq\kappa.
\]
By \Cref{thm:elementary-feature-concentration-comparison}\textup{(iii)},
$\od(X_n;-\kappa)\leq\varepsilon$, which proves \textup{(ii)}.  This
completes the proof.
\end{proof}

We finish with a discrete family that exhibits the discrepancy between
the invariants of a function family and those of its underlying mm-space, the
need for both one-sided tails, and the sharpness for this family of the
asymmetric two-radius bound in
\Cref{thm:elementary-feature-concentration-comparison}\textup{(iii)}.  For
$N\geq3$, set
\[
E_N\coloneqq\{1,\ldots,N\},\qquad
\mu_N\coloneqq\frac1N\sum_{j=1}^N\delta_j,\qquad
X_N\coloneqq(E_N,\{f_m\mid1\leq m\leq N\},\mu_N),
\]
where
\[
f_m(j)\coloneqq
\begin{cases}
0,&j=m,\\
1,&j\neq m.
\end{cases}
\]
Then $\overline{F_{X_N}}=F_{X_N}$ and
\[
d_{F_{X_N}}(i,j)=1\quad(i\neq j),\qquad
(f_m)_*\mu_N=\frac1N\delta_0+
\left(1-\frac1N\right)\delta_1.
\]
Put $M_N\coloneqq(E_N,d_{F_{X_N}},\mu_N)$.  In the right-hand column of
\Cref{tab:discrete-feature-concentration}, $\od$ and
$\operatorname{Sep}$ denote the classical observable diameter and separation
distance of the underlying mm-space $M_N$, respectively
\cite[Definitions~2.13 and~2.24]{shioya2016mmg}.  For a general mm-space
$M=(Y,d_Y,\mu_Y)$, we also write its classical concentration function as
\[
\alpha_M(r)\coloneqq
\sup\{1-\mu_Y(U_r(A))\mid
A\subset Y\text{ is Borel and }\mu_Y(A)\geq1/2\}.
\]

\begin{table}[t]
\centering
\small
\begin{tabular}{@{}p{.20\linewidth}p{.39\linewidth}p{.32\linewidth}@{}}
\hline
Item&Function-family invariant&Underlying mm-space or consequence\\
\hline
Observable diameter&
$\od(X_N;-\kappa)=1$ for $0<\kappa<1/N$, and $0$ for
$1/N\leq\kappa<1$&
For example, $\od(M_4;-1/3)=1$, whereas the function-family value is $0$\\
\hline
Separation distance\newline
$\kappa>1/N$, $2\lceil\kappa N\rceil\leq N$&
$\operatorname{Sep}(X_N;\kappa,\kappa)
=\operatorname{Sep}^+(X_N;\kappa,\kappa)=0$&
$\operatorname{Sep}(M_N;\kappa,\kappa)=1$\\
\hline
Concentration function&
For $0<r\leq1$, $\alpha_{X_N}^+(r)=0$ and
$\alpha_{X_N}^-(r)=1/N$.  Both are $0$ for $r>1$&
For $0<r\leq1$, $\alpha_{M_N}(r)=\lfloor N/2\rfloor/N$.  It is $0$ for
$r>1$\\
\hline
Asymmetric two-radius bound&
$\inf\{r_++r_-\mid
\alpha_{X_N}^+(r_+)+\alpha_{X_N}^-(r_-)\leq\kappa\}$ is $1$ for
$\kappa<1/N$ and $0$ for $\kappa\geq1/N$&
The optimal value in
\Cref{thm:elementary-feature-concentration-comparison}\textup{(iii)}
equals $\od(X_N;-\kappa)$\\
\hline
Failure of either tail alone&
For $X_3$, $\alpha_{X_3}^+(r)=0$ for $r>0$, but
$\od(X_3;-\kappa)=1$ for $0<\kappa<1/3$&
Writing $X_3^-$ for the family obtained by replacing $f_m$ with $-f_m$, we
have $\alpha_{X_3^-}^-(r)=0$ for $r>0$ and the same observable diameter\\
\hline
\end{tabular}
\caption{Function-family concentration and underlying mm-space concentration
for the discrete family $X_N$}
\label{tab:discrete-feature-concentration}
\end{table}

\appendix

\section{The coarsening adjunction}
\label{app:coarsening}
\label{sec:coarsening}

The coarsening adjunction constructed here supplies the inclusion functors
used in the pullback conventions below when the monoidal closure family is
enlarged without changing the underlying measured set.

\begin{theorem}[Inclusion and saturation]
\label{thm:coarsening-adjunction}
Let $\TB\subset\L\subset\K\subset\lipone(\R)$ be monoidal families.  The
inclusion functor
\[
J_{\K,\L}(-)\colon\K\circ\D\longrightarrow\L\circ\D
\]
has the right adjoint
\[
S_{\L,\K}(-)\colon\L\circ\D\longrightarrow\K\circ\D,\qquad
S_{\L,\K}(X)\coloneqq(X,\K\circ F_X,\mu_X).
\]
Neither functor changes underlying morphism maps.  The unit of
$J_{\K,\L}\dashv S_{\L,\K}$ is an isomorphism, so $J_{\K,\L}$ is fully
faithful.  The counit need not be an isomorphism.  For a finite example,
let $E=\{0,1,2,3\}$ have uniform measure, let
$\iota\colon E\to\R$ be the coordinate function, and set
$X=(E,\TB\circ\{\iota\},\mu_E)$.  For $\L=\TB$ and $\K=\Lup$, the
$\K$-saturation contains $\iota/2$, whereas
$\iota/2\notin\overline{\TB\circ\{\iota\}}$; hence the counit at $X$ is
not an isomorphism.
\end{theorem}

\begin{proof}
Every $\K$-gd-set is an $\L$-gd-set, so $J_{\K,\L}$ is well-defined.  Since
$\operatorname{id}_{\R}\in\K$ and every element of $\K$ is
$1$-Lipschitz, the function family $\K\circ F_X$ induces the same metric as
$F_X$.  Closure of $\K$ under composition and pointwise convergence makes
$S_{\L,\K}(X)$ a $\K$-gd-set.  If $u\colon X\to Y$ is a domination, then
\[
F_Y\circ u\subset\overline{F_X}
\quad\Longrightarrow\quad
\K\circ F_Y\circ u\subset\overline{\K\circ F_X},
\]
so saturation is functorial.

Let $Y\in\K\circ\D$, $X\in\L\circ\D$, and let $u\colon Y\to X$ be a
measure-preserving map.  It is a morphism $J_{\K,\L}(Y)\to X$ precisely
when
\begin{equation} \label{eq:coarsening-hom-condition} F_X\circ u\subset\overline{F_Y}. \end{equation}
Because $Y$ is $\K$-closed, \Cref{eq:coarsening-hom-condition} implies
$\K\circ F_X\circ u\subset\overline{F_Y}$, which says that $u$ is a
morphism $Y\to S_{\L,\K}(X)$.  The converse follows from
$\operatorname{id}_{\R}\in\K$.  This gives the natural bijection
\[
\operatorname{Hom}_{\L\circ\D}(J_{\K,\L}(Y),X)
\simeq
\operatorname{Hom}_{\K\circ\D}(Y,S_{\L,\K}(X)).
\]
For $Y\in\K\circ\D$,
$\overline{\K\circ F_Y}=\overline{F_Y}$, so the unit is the identity
isomorphism.

It remains to verify the finite example.  The identity belongs to $\TB$, so
$F_X=\TB\circ\{\iota\}$ induces the ordinary distance on $E$, and the
monoid property makes $X$ a $\TB$-gd-set.  The map $t\mapsto t/2$ belongs
to $\Lup$, and therefore $\iota/2$ belongs to the $\Lup$-saturation.  For
every $h\in\TB$, the three increments
$h(j+1)-h(j)$, $j=0,1,2$, lie in $[0,1]$.  If an increment lies strictly
between $0$ and $1$, the interval $[j,j+1]$ contains a truncation boundary.
Thus three consecutive increments cannot all lie strictly between $0$ and
$1$.  This property persists under pointwise convergence on the finite set
$E$, whereas the three increments of $\iota/2$ are all $1/2$.  Hence
$\iota/2\notin\overline{F_X}$.  The counit is the identity on the underlying
measured set, so it cannot have an inverse morphism.  This completes the
proof.
\end{proof}

\begin{proposition}[Box behavior of inclusion and saturation]
\label{prop:coarsening-box}
Let $\TB\subset\L\subset\K\subset\lipone(\R)$ be monoidal families.  For
$X,Y\in\K\circ\D$ and $X',Y'\in\L\circ\D$,
\begin{align}
\Box(J_{\K,\L}(X),J_{\K,\L}(Y))&=\Box(X,Y), \label{eq:coarsening-inclusion-isometry}\\
\Box(S_{\L,\K}(X'),S_{\L,\K}(Y'))&\leq\Box(X',Y'). \label{eq:coarsening-saturation-nonexpansive}
\end{align}
\end{proposition}

\begin{proof}
The inclusion functor changes no gd-set data, which proves
\Cref{eq:coarsening-inclusion-isometry}.  The saturation estimate
\Cref{eq:coarsening-saturation-nonexpansive} is the specialization of
\cite[Lemma~4.6]{gds2} to $S_{\L,\K}$.
	This completes the proof.
\end{proof}

Domination refinement
\cite[Lemma~5.4]{gds2} and
\Cref{thm:coarsening-adjunction,prop:coarsening-box} show that
$J_{\K,\L}\dashv S_{\L,\K}$ is a pyramidal adjunction
\cite[Definition~A.5]{gds3-ja}.  Its pyramid transport and compatibility
under finite composition are therefore the specializations of
\cite[Theorem~A.7 and Proposition~A.8]{gds3-ja}.

If $\mathcal Q$ is a pyramid in $\K\circ\D$, its transported pyramid is
\[ (J_{\K,\L})_{\#}\mathcal Q \coloneqq\{Y\in\L\circ\D\mid S_{\L,\K}(Y)\in\mathcal Q\}. \]
Then
\begin{equation}
\label{eq:coarsening-pyramid-measurement}
\M((J_{\K,\L})_{\#}\mathcal Q;N,R)=\M(\mathcal Q;N,R).
\end{equation}
Indeed, the unit is an isomorphism, so every $Z\in\mathcal Q$ contributes
the same measurements through $J_{\K,\L}(Z)$.  Conversely, if
$Y\in(J_{\K,\L})_{\#}\mathcal Q$, then $S_{\L,\K}(Y)\in\mathcal Q$ and
\[
F_Y\subset\K\circ F_Y=F_{S_{\L,\K}(Y)}.
\]
Thus
$\M(Y;N,R)\subset\M(S_{\L,\K}(Y);N,R)
\subset\M(\mathcal Q;N,R)$, proving the reverse inclusion.

We use the following pullback notation.  Let
$L\colon\mathcal C\to\TB\circ\D$ be a functor and let
$C\in\mathcal C$.  Set
\[
\od_L(C;-\kappa)\coloneqq\od(L(C);-\kappa).
\]
Also set
\[
\operatorname{Sep}_L(C;\kappa_0,\ldots,\kappa_N)
\coloneqq\operatorname{Sep}(L(C);\kappa_0,\ldots,\kappa_N).
\]
For a nonempty family $\mathcal E$ in $\mathcal C$, set
\[ \od_L(\mathcal E;-\kappa)\coloneqq \sup_{C\in\mathcal E}\od_L(C;-\kappa). \]
Likewise, set
\[
\operatorname{Sep}_L(\mathcal E;\kappa_0,\ldots,\kappa_N)
\coloneqq\sup_{C\in\mathcal E}
\operatorname{Sep}_L(C;\kappa_0,\ldots,\kappa_N).
\]
When $\TB\subset\mathcal L$ and
$L\colon\mathcal C\to\mathcal L\circ\D$, the subscript $L$ denotes the
composite $J_{\mathcal L,\TB}\circ L$.

\begin{theorem}[Pullback limits along pyramidal adjunctions]
\label{thm:observable-diameter-pyramid-limit}
\label{thm:multi-directed-separation-pyramid-limit}
Let $\mathcal C$ be a $\Box$-metrized category, and let
$\mathcal P_n\to\mathcal P$ weakly be pyramids in $\mathcal C$.
\begin{enumerate}
\item If
\[
L\colon\mathcal C\rightleftarrows\Lup\circ\D\colon R
\]
is a pyramidal adjunction, then, for every $\kappa\in(0,1)$,
\begin{align*}
\od_L(\mathcal P;-\kappa)
&=\lim_{\varepsilon\downarrow0}\liminf_{n\to\infty}
\od_L(\mathcal P_n;-(\kappa+\varepsilon))\\
&=\lim_{\varepsilon\downarrow0}\limsup_{n\to\infty} \od_L(\mathcal P_n;-(\kappa+\varepsilon))
\end{align*}
in $[0,+\infty]$, where $0<\varepsilon<1-\kappa$.
\item Suppose that
\[
L\colon\mathcal C\rightleftarrows\TB\circ\D\colon R
\]
is a pyramidal adjunction.

For every positive integer $N$ and
$\kappa_0,\ldots,\kappa_N\in(0,1)$,
\begin{align*}
\operatorname{Sep}_L(\mathcal P;\kappa_0,\ldots,\kappa_N)
&=\lim_{\varepsilon\downarrow0}\liminf_{n\to\infty}
\operatorname{Sep}_L(\mathcal P_n;\kappa_0-\varepsilon,\ldots,
\kappa_N-\varepsilon)\\
&=\lim_{\varepsilon\downarrow0}\limsup_{n\to\infty}
\operatorname{Sep}_L(\mathcal P_n;\kappa_0-\varepsilon,\ldots,
\kappa_N-\varepsilon)
\end{align*}
in $[0,+\infty]$, where
$0<\varepsilon<\min_{0\leq j\leq N}\kappa_j$.
\end{enumerate}
\end{theorem}

\begin{proof}[Proof of \textup{(1)}]
The weak-convergence equivalence
\cite[Theorem~A.7]{gds3-ja} transfers the weak-limit formula for an
isomorphism-invariant, domination-monotone functional along a pyramidal
adjunction, with the same parameter perturbation.  Observable diameter has
these two properties.  Applying this argument to
\Cref{prop:lup-observable-diameter-pyramid-limit} proves the theorem with
$\kappa\mapsto\kappa+\varepsilon$.
	This completes the proof.
\end{proof}

\begin{proof}[Proof of \textup{(2)}]
Use the pullback principle stated in the proof of \textup{(1)}, replacing
observable diameter by separation and $\kappa+\varepsilon$ by the componentwise
perturbation $\kappa_j-\varepsilon$.  Separation is invariant under isomorphism
and monotone under domination, while
\Cref{prop:tb-multi-directed-separation-pyramid-limit} supplies the target
formula.
	This completes the proof.
\end{proof}

The right adjoints recover a distance from a gd-set $X$ on the same measured
set: $\operatorname{Rec}(X)\coloneqq(X,d_{F_X},\mu_X)$, and
$\operatorname{Rec}^{+}(X)\coloneqq(X,d_X^{+},\mu_X)$ with
\[d_X^{+}(x,x')\coloneqq\sup_{f\in\overline{F_X}}\{f(x')-f(x)\}.\]
The representation adjunctions $\operatorname{Rep}\dashv\operatorname{Rec}$ and
$\operatorname{Rep}^{+}\dashv\operatorname{Rec}^{+}$ are pyramidal
\cite[Corollary~A.12]{gds3-ja}.  Their composites with the inclusion
adjunctions above remain pyramidal by
\cite[Proposition~A.8]{gds3-ja}.  Therefore
\Cref{thm:observable-diameter-pyramid-limit} applies to mm-space and qm-space
pyramids.

For consistency with qm-space and mm-space invariants, let $X$ be a qm-space
and put $d_X(A,B)\coloneqq\inf_{a\in A,\,b\in B}d_X(a,b)$.  Every
$f\in\lipplus(X)$ satisfies $f(b)-f(a)\leq d_X(a,b)$.  Conversely,
\[
p_A(x)\coloneqq\inf_{a\in A}d_X(a,x)
\]
belongs to $\lipplus(X)$, vanishes on $A$, and is at least $d_X(A,B)$ on
$B$.  Therefore
\begin{equation}
\label{eq:qms-directed-separation-computation}
\sigma^{\to}(\operatorname{Rep}^{+}(X);A,B)=d_X(A,B).
\end{equation}
Since $\overline{F_{\operatorname{Rep}^{+}(X)}}=\lipplus(X)$, observable diameter is
computed over the one-sided Lipschitz functions, while separation is computed
from the directed set distances in
\Cref{eq:qms-directed-separation-computation}.  If $X$ is an mm-space, then
\Cref{eq:symmetric-directed-lipschitz} reduces these quantities to the
	classical observable diameter and separation distance
	\cite[Definitions~2.13 and~2.24]{shioya2016mmg}.  Thus the
gd-set invariants and their pullbacks agree with the pre-existing qm-space and
mm-space observable diameter and separation distance.

\section*{Acknowledgments}

The author would like to thank Professor Takashi Shioya for many helpful
suggestions and guidance.  The author used Claude and GPT-5.6-series Codex
models as AI-assisted tools in preparing this manuscript.  The author
reviewed and revised the mathematical content and takes full responsibility
for the final manuscript.

\end{document}